\documentclass[pdflatex,sn-mathphys]{sn-jnl}% Math and Physical Sciences Numbered Reference Style
\usepackage{graphicx}%
\usepackage{multirow}%
\usepackage{amsmath,amssymb,amsfonts}%
\usepackage{amsthm}%
\usepackage{mathrsfs}%
\usepackage[title]{appendix}%
\usepackage{xcolor}%
\usepackage{textcomp}%
\usepackage{manyfoot}%
\usepackage{booktabs}%
\usepackage{algorithm}%
\usepackage{algorithmicx}%
\usepackage{algpseudocode}%
\usepackage{listings}%
\theoremstyle{thmstyleone}%
\newtheorem{theorem}{Theorem}%  meant for continuous numbers
\newtheorem{proposition}[theorem]{Proposition}% 

\newtheorem{corollary}{Corollary}

\theoremstyle{thmstyletwo}%
\newtheorem{remark}{Remark}%

\theoremstyle{thmstylethree}%
\newtheorem{definition}{Definition}%

\begin{document}

\title[On higher-order differential equations for Laguerre-Hahn orthogonal polynomials]{On higher-order differential equations for Laguerre-Hahn orthogonal polynomials}


\author[1]{\fnm{Mohamed} \sur{Khalfallah}}\email{mohamed.khalfallah@fsm.rnu.tn}

\author*[2]{\fnm{Zélia} \sur{da Rocha}}\email{mrdioh@fc.up.pt}

\affil[1]{\orgdiv{Department of Mathematics}, 
\orgname{Faculty of Sciences of Monastir, University of Monastir}, 
\orgaddress{
%\street{}, 
\city{Monastir}, 
\postcode{5019}, 
%\state{}, 
\country{Tunisia}}}

\affil[2]{\orgdiv{Departamento de Matemática, Centro de Matemática da Universidade do Porto (CMUP)}, 
\orgname{Faculdade de Ciências da Universidade do Porto}, 
\orgaddress{
\street{Rua do Campo Alegre n. 687}, 
\city{Porto}, 
\postcode{4169-007}, 
%\state{}, 
\country{Portugal}}}


\abstract{In this work, we present a constructive method to derive homogeneous linear differential equations of arbitrary order for Laguerre--Hahn orthogonal polynomials. The approach is based on a hierarchy of structure relations, built recursively from the fundamental structure relation of the Laguerre--Hahn forms. The fourth-order differential equation, which has been extensively studied in the literature, is naturally recovered as a particular case of this general framework. Special attention is devoted to the semiclassical case. The method is fully algorithmic and has been implemented in {\it Mathematica$^{\circledR}$}. We apply it to two Laguerre--Hahn families of class zero analogous to Hermite, providing explicit structure relations and differential equations of orders five and ten. We also examine the classical Hermite sequence, which is recovered as a limiting case of the first family.}

\keywords{Orthogonal polynomials, Laguerre-Hahn forms, higher-order differential equations, structure relations, Hermite polynomials, algorithms, symbolic computations}


\pacs[MSC 2020 Classification]{34, 33C45, 33D45, 42C05, 33F10, 68W30, 62-09, 33F05, 65D20, 68-04}

\maketitle

%%% ----------------------------------------------------------------------
%\maketitle
%%% ----------------------------------------------------------------------
%\tableofcontents%%%%%%%%%%%%%%%%%%%%%%%%%%%%%%%%
%\tableofcontents

%\newpage
%%%%%%%%%%%%%%%%%%%%%%%%%%%%%%%%%%%
\section{Introduction}
%%%%%%%%%%%%%%%%%%%%%%%%%%%%%%%%%%%

This work is a direct continuation and extension of our recent article \cite{Article-1-NA}, in which we developed a constructive method for deriving the fourth-order homogeneous linear differential equation satisfied by any Laguerre--Hahn orthogonal polynomial sequence.\\

The investigation of orthogonal polynomial sequences \(\{P_n\}_{n\ge 0}\) satisfying linear differential equations with polynomial coefficients of bounded degrees is a classical topic in the theory of special functions, with deep connections to spectral analysis, operator theory, and approximation theory \cite{Buendia-1988}. We consider differential equations of the form
\[
\sum_{k=0}^{N} f_k(x; n) \, y^{(k)}(x) = 0,
\]
where the coefficients \(f_k\) are polynomials in \(x\), possibly depending on the degree parameter \(n\), and \(N\) denotes the order of the equation. Among the most celebrated results in this area is the characterization of the classical families—Hermite, Laguerre, Jacobi, and Bessel—by second-order differential equations of the type
\[
\sigma(x) y''(x) + \tau(x) y'(x) + \lambda_n y(x) = 0,
\]
where \(\sigma\) and \(\tau\) are fixed polynomials with \(\deg \sigma \le 2\) and \(\deg \tau = 1\), and \(\lambda_n\) is a spectral parameter depending linearly on \(n\). This fundamental result goes back to Bochner \cite{Bochner-1929} and Hahn \cite{Hahn-1978, Hahn-1983}, and it completely classifies orthogonal polynomial sequences
satisfying a second-order differential equation with polynomial coefficients.\\

A natural and challenging extension of this theory is the study of orthogonal polynomials satisfying higher-order differential equations. For instance, it was proved in \cite{KWON-1996,Loureiro-2006} that an orthogonal polynomial sequence is classical if and only if it fulfils a certain differential equation of order $2k$, for $k\geq 1$.
It is worth mentioning that it was established in \cite{Hahn-1978} that whenever an orthogonal polynomial sequence satisfies a differential equation with polynomial coefficients, the order can always be reduced to the minimal cases \(N=2\) or \(N=4\). The second-order case being fully understood, the fourth-order case becomes the next frontier. This led to the introduction of the Laguerre--Hahn class, a broad family of orthogonal polynomials defined by means of a Riccati equation satisfied by the formal Stieltjes function $S$ \cite{Alaya-these-1996, Alaya-Maroni,Dini-these-1988,Dzoumba-these-1985,Magnus-1983, Maroni-1983, Maroni-1991},
\begin{equation}\label{SReq}
\Phi S' = B S^2 + C S + D, \qquad \Phi \not\neq 0,
\end{equation}
where \(\Phi, B, C, D\) are polynomials. This class was extensively developed by Maroni and his collaborators
 through a detailed algebraic and topological formalism  \cite{Alaya-these-1996,Alaya-Maroni,Bouakkaz-these,Bouakkaz-Maroni-1991,Dini-these-1988,Dzoumba-these-1985,Maroni-1991}. It occupies a prominent position in the theory of orthogonal polynomials, as it encompasses not only the classical and semiclassical families, but also many sequences obtained by association \cite{Askey-1984,Belmehdi-Ronveaux-1991,Bustoz-1982}, and by perturbations, such as co-recursive, co-dilated, and co-modified polynomials \cite{Letessier-1994,Ronveaux-1990,Wimp-1987}. Despite its richness and the variety of known examples, the general classification of Laguerre--Hahn orthogonal polynomials remains largely open.\\

For strict Laguerre--Hahn polynomials, i.e., non-semiclassical ones, the minimal order of a differential equation is four. 
However,  in contrast with the second-order case, such differential equations do not yet yield a complete characterization of the Laguerre-Hahn class, since the converse implication is still open in general. It is conjectured that every orthogonal polynomial sequence satisfying a fourth-order differential equation with polynomial coefficients of bounded degrees necessarily belongs to the Laguerre-Hahn class (see \cite{Magnus-1983}, \cite[Sect.III]{brezinski1985polynomes}).

Several representations of this fourth-order equation have been given over the years: as a determinant of order five in \cite{Dzoumba-these-1985}, as a determinant of order three in \cite{Dini-these-1988}, and in explicit polynomial form in \cite{Bouakkaz-these}. In our recent work \cite{Article-1-NA}, we developed a constructive method based on four structure relations and algebraic eliminations, yielding the fourth-order equation as a determinant of order four. This approach combines a compact representation of the differential equation with a systematic constructive procedure, leading to the symbolic algorithm  {\it 4oDELH} implemented in {\it Mathematica$^{\circledR}$}. This algorithm was applied to class-zero Laguerre--Hahn families analogous to Hermite, providing the explicit fourth-order equations for these families for the first time \cite{Article-1-NA}. 
Subsequently, in \cite{Article-2-Soummi}, a complete description of the class-zero Laguerre--Hahn families was established, updating and correcting the one previously presented in \cite{Bouakkaz-Maroni-1991}. Later, in \cite{Article-3-Soummi}, we developed a recursive procedure to compute the moments of all these canonical class-zero families.\\

The present paper addresses a natural question that has not been explored before: \emph{can the method of \cite{Article-1-NA} be extended to yield differential equations of arbitrary order \(N \ge 4\)?} We answer this question affirmatively. In doing so, we uncover a recursive structure that underlies the theory of Laguerre--Hahn orthogonal polynomials, revealing that the fourth-order equation is but the first step in a hierarchy of differential equations of all orders.\\

The main contribution of this work is the construction, for each integer \(i \ge 1\), of a hierarchy of structure relations of the form
\[
G_{0,i}(x;n) P_{n-1}^{(1)}(x) + G_{1,i}(x;n) P_n^{(1)}(x) + H_i(x;n) P_n(x)
= \sum_{k=0}^{i} M_{k,i}(x;n) P_{n+1}^{[k]}(x),
\]
where \(\{P_n^{(1)}\}_{n\ge 0}\) is the associated polynomial sequence of order one of the monic orthogonal sequence \(\{P_n\}_{n\ge 0}\), and \(P_{n+1}^{[k]}\) denotes the \(k\)-th derivative of \(P_{n+1}\). The coefficients are computed recursively from the data of the Stieltjes equation of Riccati type \eqref{SReq} and the recurrence coefficients of the sequence; this is stated in Theorem~\ref{th:0.1}. From this hierarchy, by selecting four relations with indices $i_1$, $i_2$,  $i_3$, and  $i_4$, such that \(1\leq i_1 < i_2 < i_3 < i_4 = N\), we derive a homogeneous linear differential equation of order \(N\). This result is stated in Theorem~\ref{th:0.2}.\\

This general framework yields several new results. For \(N=4\), with the natural choice \((i_1,i_2,i_3,i_4)=(1,2,3,4)\), we recover the fourth-order equation established in \cite{Article-1-NA}, thus validating the extension. In the semiclassical setting, corresponding to $B=0$ in equation \eqref{SReq}, the recurrence simplifies and yields $N-1$ distinct differential equations of arbitrary order \(N \ge 2\). For \(N>4\), the method offers multiple choices of indices, leading, for example, to four distinct equations of order  \(N=5\).

The method is fully algorithmic and has been implemented in {\it Mathematica$^{\circledR}$} as an extension of our previous algorithm, named {\it 4oDELH}; the new implementation is called {\it HoDELH} ({\it higher-order differential equation for Laguerre--Hahn}). The implementation allows the automatic computation of structure relations and differential equations of any prescribed order, together with the reduction of common factors in the coefficients of differential equations.\\

To illustrate the effectiveness of our approach, we apply it to the two Laguerre--Hahn families of class zero analogous to Hermite, for which we provide explicit structure relations and differential equations of prescribed orders. The classical Hermite sequence is recovered as a limiting case of the first family, and further results are also presented.\\

The paper is organized as follows. In Section~\ref{Section2}, we recall the necessary background on orthogonal polynomials, moment functionals, and Laguerre--Hahn forms. Section~\ref{Section3} contains the main theoretical results: the recurrence theorem for structure relations (Theorem~\ref{th:0.1}) and the general differential equation of order \(N\) (Theorem~\ref{th:0.2}). The semiclassical case is treated in Subsection~\ref{Section3.1}. In Section~\ref{Section4}, we present the algorithm {\it HoDELH}. Section~\ref{Section5} is devoted to applications to the Laguerre–
Hahn families of class zero analogous to Hermite, with explicit results for orders $N=5$ and $N=10$. Special attention is given to the classical Hermite case, where new results are presented. The final section contains some conclusions.

%%%%%%%%%%%%%%%%%%%%%%%%%%%%%%%%%%%
\section{Notation and basic background}\label{Section2} 
%%%%%%%%%%%%%%%%%%%%%%%%%%%%%%%%%%%

In this section, we present some basic definitions, notations, and results that are used throughout this paper.
%%%%%%%%%%%%%%%%%%%%%%%%%%%%%%%%%%%%%%%%%%%%%%%%%%%%%%%%%
\subsection{Basic tools}
Let $\mathcal{P}$ denote the vector space of polynomials with complex coefficients, and let $\mathcal{P'}$ be its algebraic dual space.  
The elements of $\mathcal{P'}$ will be referred to as \emph{forms} (or linear functionals).  
The pairing between $\mathcal{P}$ and $\mathcal{P'}$ is expressed through the duality brackets $\langle \cdot, \cdot \rangle$.  
For a form $u \in \mathcal{P'}$, the sequence of complex numbers $(u)_n,~ n \geq 0$, is called the \emph{moment sequence} of $u$ relative to the monomial basis $\{x^n\}_{n \geq 0}$.  
In particular, the $n$-th moment is given by $(u)_n := \langle u, x^n \rangle$,  so that $u$ is uniquely determined by the sequence of its moments.  

In the following, we shall refer to a sequence $\{P_n\}_{n \geq 0}$ as a \emph{polynomial sequence} (PS) if $\deg P_n = n$ for all $n \geq 0$. A \emph{monic polynomial sequence} (MPS) is a PS in which each polynomial has a leading coefficient equal to one.  
If ${\{P_{n}\}}_{n \geq 0}$ is a MPS, there exists a unique sequence
$\{u_n\}_{n\geq 0}$, $u_n\in\mathcal{P}^{\prime}$, called the dual sequence of $\{P_{n}\}_{n\geq 0}$, such that,
\begin{equation}\label{SucDual}
\langle u_{n},P_{m}\rangle=\delta_{n,m}, \quad n,m\geq 0.
\end{equation}
We say that a sequence of forms $\{v_n \}_{ n \geq 0}$ is normalised if and only if $\left( v_n  \right)_n = 1, $ $n \geq 0$, and if $n \geq 1$, then $\left( v_n  \right)_m = 0, $ $m=0, \ldots, n-1$. Thus, the dual sequence  $\{u_n\}_{n\geq 0}$ is normalised. The first form $u_0$ is called the canonical form of ${\{P_{n}\}}_{n \geq 0}$.

We now introduce some operations on $\mathcal{P'}$ following \cite{Maroni-1991}.  
For $c \in \mathbb{C},~ f,p \in \mathcal{P}$, and $u \in \mathcal{P'}$, we define
\begin{align*}
&\langle fu, p \rangle = \langle u, fp \rangle, \quad
\langle u', p \rangle = -\langle u, p' \rangle,  \\
&\langle (x-c)^{-1}u, p \rangle = \langle u, \theta_c p \rangle
= \left\langle u, \frac{p(x)-p(c)}{x-c} \right\rangle.
\end{align*}
Given $f \in \mathcal{P}$ and $u \in \mathcal{P'}$, the product $uf$ is defined by  
$$(u f)(x) := \left\langle u, \displaystyle\frac{x f(x)-\zeta f(\zeta)}{x-\zeta} \right\rangle .$$\\
This definition allows us to introduce the \emph{Cauchy product} of two forms $u,v \in \mathcal{P'}$ by 
\[
\langle uv, f \rangle := \langle u, v f \rangle, \quad f \in \mathcal{P}.
\]

In addition, we make use of the formal Stieltjes function associated with 
$u \in \mathcal{P}'$, defined by  \cite{Maroni-1991}
\[
S(u)(z) := - \sum_{n \geq 0} \frac{(u)_n}{z^{n+1}},
\]
which provides an alternative representation of the moment sequence $\{(u)_n\}_{n \geq 0}$.  
Since the moments uniquely determine $u$, the function $S(u)(z)$ does so as well.

\medskip

A linear functional $u$ is called \emph{regular} (or \emph{quasi-definite}) if there exists a sequence of polynomials  $\{P_n\}_{n \geq 0}$ such that \cite{Chihara-1978}
\begin{equation*}
\langle u, P_n P_m \rangle = r_n \, \delta_{n,m}, \quad n, m \geq 0,
\end{equation*}
where $\{r_n\}_{n \geq 0}$ is a sequence of nonzero complex numbers and $\delta_{n,m}$ denotes the Kronecker symbol.
The sequence $\{P_n\}_{n\geq 0}$ is then said orthogonal with respect to $u$. 
Then, necessarily, $\{P_n\}_{n \geq 0}$ is a PS, $u=(u)_0u_0$, and $\{P_n\}_{n \geq 0}$ and $u$ can be normalized. In the sequel, we shall consider that each $P_n(x)$ is monic, and $(u)_0=1$ (i.e. $u=u_0$).
Henceforth, a monic orthogonal
polynomial sequence $\{P_n\}_{n\geq 0}$ will be indicated as MOPS. 

It is well known that an MOPS is characterized by the following second-order linear recurrence relation and initial conditions,  \cite{Chihara-1978}
\begin{eqnarray}
&& P_0(x)=1\  ,\quad  P_1(x)=x-\beta_0, \label{ic_TTRR}\\
&& P_{n+2}(x)=(x-\beta_{n+1})P_{n+1}(x)-\gamma_{n+1}P_{n}(x),~~n\geq 0, \label{TTRR}
\end{eqnarray}
being $\{\beta_n\}_{n\geq 0}$ and $\{\gamma_{n+1}\}_{n\geq 0}$ sequences of complex numbers such that $\gamma_{n+1}\neq0$ for all $ n\geq 0$.

Let  $\{P_n^{(1)}\}_{n\geq0}$ be the associated polynomial sequence of order one of the MPS  $\{ P_n\}_{n\geq0}$ with respect to the canonical form $u=u_0$. It is well known that \cite{Chihara-1978}
$$
 P_n^{(1)}(x):=(u\theta_0P_{n+1})(x)=\bigg\langle  u,\frac{P_{n+1}(x)- P_{n+1}(\xi)}{x-\xi}  \bigg\rangle.
$$
The Stieltjes function of $u^{(1)}$ is expressed in terms of that of $u$ as  \cite{Maroni-1991} 
$$
\gamma_1 S\left(u^{(1)}\right)(z)=-\frac{1}{S(u)(z)}-\left(z-\beta_0\right).
$$
More generally, the sequence of associated polynomials of order $(r+1)$, $r\geq 1$, is  defined by recursion
$$
P_n^{(r+1)}(x)=\big(  P_n^{(r)} \big)^{(1)}(x),\quad  u_n^{(r+1)}=\big(  u_n^{(r)} \big)^{(1)},\quad n, r \geq 0.
$$

If $\{ P_n\}_{n\geq0}$ is a MOPS with respect to the form $u$, then for $r\in\mathbb{N}$, the associated sequence of polynomials of order $r$,  $\left\{ P_n^{(r)}\right\}_{n\geq0}$, is also orthogonal with respect to the form $u^{(r)}$ and satisfies the following recurrence relation 
\begin{eqnarray}
&&P_0^{(r)}(x)=1, \quad P_1^{(r)}(x)=x-\beta_0^{(r)},\label{ic_ASSTTRR}\\
&&P_{n+2}^{(r)}(x)=(x-\beta_{n+1}^{(r)})P_{n+1}^{(r)}(x)-\gamma_{n+1}^{(r)}P_n^{(r)}(x),\quad n\geq0,  \label{ASSTTRR}
\end{eqnarray}
where 
\begin{equation*}
\beta_{n}^{(r)}=\beta_{n+r}, \quad \gamma_{n+1}^{(r)}=\gamma_{n+1+r},\quad n\geq0.  
\end{equation*}

We recall the definition of the $r$-perturbed sequence $\{\widetilde{P}_n\}_{n\geq 0}$, for a fixed integer $r\geq 0$, associated with a MOPS $\{P_n\}_{n\geq 0}$, as introduced in \cite{Maroni-1991}. It is a MOPS satisfying the following second-order recurrence relation 
\[
\begin{array}{l}
\widetilde{P}_0(x)=1,\quad \widetilde{P}_1(x)=x-\widetilde{\beta}_0,\\[2pt]
\widetilde{P}_{n+2}(x)=(x-\widetilde{\beta}_{n+1})\widetilde{P}_{n+1}(x)-\widetilde{\gamma}_{n+1}\widetilde{P}_n(x),\quad n\geq0,
\end{array}
\]
with
\[
\begin{aligned}
&\widetilde{\beta}_0 = \beta_0 + \mu_0,\\[2pt]
&\widetilde{\beta}_n = \beta_n + \mu_n,\quad \mu_n\in\mathbb{C};\qquad 
\widetilde{\gamma}_n = \lambda_n\gamma_n,\quad \widetilde{\gamma}_n\in\mathbb{C}\setminus\{0\},\quad 1\leq n\leq r,\\[2pt]
&\widetilde{\beta}_n = \beta_n,\quad \widetilde{\gamma}_n = \gamma_n,\quad n\geq r+1.
\end{aligned}
\]
We assume that either $\mu_r\neq 0$ or $\lambda_n\neq 1$. The so-called co-recursive case corresponds to a perturbed case of order $0$. Using the notations $\mu:=(\mu_1,\ldots,\mu_r)$, $\lambda:=(\lambda_1,\ldots,\lambda_r)$, $r\geq 1$, we write
\[
\widetilde{P}_n(x)=P_n\left(\mu_0;\,{\mu\atop\lambda}\,;r;x\right),\qquad n\geq0,
\]
and the sequence $\{\widetilde{P}_n\}_{n\geq0}$ is orthogonal with respect to the perturbed form 
$$\widetilde{u}:=u\left(\mu_0;\,{\mu\atop\lambda}\,;r\right).$$

%%%%%%%%%%%%%%%%%%%%%%%%%%%%%%%%%%%
\subsection{Laguerre-Hahn forms}
%%%%%%%%%%%%%%%%%%%%%%%%%%%%%%%%%%%

\begin{definition} \cite{Dzoumba-these-1985,Magnus-1983,Maroni-1983}
A regular form $u$, with $(u)_0=1$, is said to be a Laguerre-Hahn form if its formal Stieltjes function satisfies the Riccati equation
\begin{equation}\label{Riccati}
A(z)S'(u)(z)=B(z)S^2(u)(z)+C(z)S(u)(z)+D(z),
\end{equation}
where $A$, $B$, $C$, and $D$ are polynomials.\\
The sequence \(\{P_n\}_{n\geq 0}\) orthogonal with respect to \( u \) is also called a Laguerre-Hahn sequence. 
\end{definition}

\begin{remark} \cite{Maroni-1991}
If $A=0$ identically, the form $u$ is classified as a second-degree form. 
If $A$ is not identically zero, it may be assumed, without loss of generality, 
that it is monic; and we let $A:=\Phi$. Under this normalization, the condition $B \neq 0$ 
characterizes $u$ as a strict Laguerre-Hahn form, whereas the case $B = 0$ 
corresponds to a semiclassical form.
\end{remark}

There are several characterizations of Laguerre-Hahn forms. Some of them are listed in the following result.
\begin{proposition} \cite{Alaya-Maroni,Bouakkaz-Maroni-1991,Dini-these-1988,Maroni-1991} 
Let $u$ be a regular and normalized form, i.e.,
 $(u)_0=1$, and let $\{P_n\}_{n\geq 0}$ be its corresponding MOPS. The following statements are equivalent
\begin{enumerate}
\item[(i)] $u$ is a Laguerre-Hahn form satisfying \eqref{Riccati} with $A=\Phi$.
\item[(ii)]  \cite{Dini-these-1988} $u$ satisfies the functional equation
\begin{equation}\label{Laguerre-Hahn-EF}
(\Phi u)'+\psi u+B(x^{-1}u^2)=0,
\end{equation}
where $\Phi$, $B$, $C$, and $D$ are the polynomials in \eqref{Riccati} and
\begin{align*}
&C=-\Phi'-\psi,\\
&D=-(u\theta_0\Phi)'-(u\theta_0\psi)-(u^2\theta_0^2B).
\end{align*}
\item[(iii)] \cite{Dini-these-1988} Each polynomial \(P_n, n \geq 0\), verifies the so-called structure relation
\begin{equation}\label{Dini_St_Rel}
\Phi(x)P_{n+1}^{\prime}(x) - B(x)P_n^{(1)}(x) = \sum_{\mu=n-s}^{n+d} \theta_{n,\mu} P_\mu(x), \quad n \geq s + 1,
\end{equation}
where \(\Phi\) and \(B\) are the polynomials defined in (i),  \(\{P_n^{(1)}\}_{n \geq 0}\) is the sequence of associated orthogonal polynomials of order 1 of \(\{P_n\}_{n \geq 0}\),  
$d=\max(t,q)$, $s=\max(p-1,d-2)$, being $t$, $p$, and $q$ the degrees of  $\Phi$, $\psi$, and $B$, respectively. 
\end{enumerate}
\end{proposition}

%%%%%%%%%%%%%%%%%%%%%%%%%%%%%%%%%%
It is worth noting that the above functional equation \eqref{Laguerre-Hahn-EF} is not uniquely determined.
Indeed, if $u$ is a Laguerre--Hahn form and $\chi$ is an arbitrary polynomial, then $u$ also satisfies
\[
(\chi \Phi u)' + \bigl(\chi \psi - \chi' \Phi\bigr) u
+ (\chi B)\,(x^{-1}u^{2}) = 0.
\]
This observation motivates the following definition.
\begin{definition} \cite{Alaya-Maroni, Bouakkaz-Maroni-1991}
The class of a Laguerre-Hahn form $u$ is the non-negative integer number defined as
$$
s:=\min\max\big\{\deg{\psi}-1, \max\{\deg{\Phi}, \deg{B}\}-2\big\},
$$
where the minimum is taken among all polynomials $\Phi, \psi$ and $B$ such that $u$
satisfies \eqref{Laguerre-Hahn-EF}.
\end{definition}

Taking into account that the class of a Laguerre-Hahn form is crucial to state a hierarchy of such families, we need to give a criterion to characterize it.
\begin{proposition}\label{proposition-simplification} \cite{Alaya-Maroni,Bouakkaz-Maroni-1991}
Let $u$ be a Laguerre-Hahn form and let $\Phi$ and $\psi$ be non-zero polynomials 
such that \eqref{Laguerre-Hahn-EF} holds.
 Let
 \begin{equation}\label{s_class_LH}
 s=\max\big\{\deg{\psi}-1, \max\{\deg{\Phi}, \deg{B}\}-2\big\}.
 \end{equation} 
 Then $s$ is the class of $u$ if and only if
\begin{equation*}
\prod_{c\in\mathcal{Z}_{\Phi}}{\Big(|\Phi'(c)+\psi(c)|+|B(c)|+|\langle u, \theta_c^2\Phi+\theta_c\psi+u\theta_0\theta_c B\rangle|\Big)}\neq0,
\end{equation*}
where $\mathcal{Z}_{\Phi}$ denotes the set of zeros of $\Phi$.
\end{proposition}
\begin{remark}
When it is possible to simplify by the factor \( x - c \), we obtain the new functional equation
\[ 
((\theta_c\Phi)u)' + (\theta_c\psi + \theta_c^2\Phi)u + (\theta_cB)(x^{-1}u^2) = 0. 
\]
Then \( u \) is of class less than or equal to \( s - 1 \).    
\end{remark}

Based on Proposition~\ref{proposition-simplification}, one obtains an alternative criterion to determine the class using the polynomials involved in the Riccati equation~\eqref{Riccati}.
\begin{corollary} \cite{Alaya-Maroni}
Let $u$ be a Laguerre-Hahn form and let $A=\Phi$, $B$, $C$, and $D$ be non-zero polynomials satisfying \eqref{Riccati}. Then $s$ given by \eqref{s_class_LH} is the class of $u$ if and only if the polynomials $\Phi$, $B$, $C$, and $D$ are coprime or, equivalently,
$$
\prod_{c\in\mathcal{Z}_{\Phi}}{\big(|B(c)|+|C(c)|+|D(c)|\big)}\neq0.
$$
\end{corollary}

%%%%%%%%%%%%%%%%%%%%%%%%%%%%%%%%%%

Before proceeding further, we recall several relations that will be essential for the developments that follow.

\begin{proposition}\label{proposition-FSR} \cite{Dini-these-1988,Dzoumba-these-1985,Maroni-1991}
Let \(\{P_n\}_{n \geq 0}\) be a MOPS with respect to \(u\), satisfying \eqref{ic_TTRR}-\eqref{TTRR}. The following statements are equivalent.
\begin{enumerate}
\item[(i)] \(u\) is a Laguerre-Hahn form of class \(s\) satisfying the functional equation \eqref{Laguerre-Hahn-EF}.

\item[(ii)] \(\{P_n\}_{n \geq 0}\) satisfies  the following structure relation 
\begin{align}
\Phi(x) P_{n+1}'(x) - B_0(x) P_n^{(1)} (x) =& \frac{1}{2}\big(C_{n+1}(x) - C_0(x)\big) P_{n+1} (x)\nonumber\\ 
&- \gamma_{n+1} D_{n+1} (x) P_n (x), \quad n \geq 0, \label{R4}    
\end{align}
where \(C_n\) and \(D_n\) are polynomials with coefficients depending on \(n\), such that 
$$\deg C_n \leq s + 1, \qquad \deg D_n \leq s,$$ 
satisfying the recurrence relations
\begin{align}
C_{n+1}(x) =& -C_n(x) + 2(x - \beta_n) D_n(x), \label{SR-1}\\
\gamma_{n+1} D_{n+1} (x) =& -\Phi(x) + \gamma_n D_{n-1} (x) - (x - \beta_n) C_n (x) + (x - \beta_n)^2 D_n (x), \label{SR-2}
\end{align}
for every \(n \geq 0\), with the initial conditions 
$$
B_0(x)=B(x),\quad C_0(x)=C(x),\quad D_0(x)=D(x),\quad \gamma_0D_{-1}(x)=B(x).
$$
\end{enumerate}
\end{proposition} 

\begin{proposition}\label{Lemma-P} \cite{Dini-these-1988,Dzoumba-these-1985}
Let \(\{P_n\}_{n \geq 0}\) be a Laguerre-Hahn MOPS. For every $n\geq0$, we have
\begin{align}
\Phi(x)(P^{(1)}_{n-1}(x))' =& D_{n}(x)P^{(1)}_{n}(x) - \frac{1}{2}\left(C_{n+1}(x) - C_0(x)\right)P^{(1)}_{n-1}(x) - D_0(x)P_{n}(x), \label{R1}\\
\Phi(x)(P_{n}(x))' =& D_{n}(x)P_{n+1}(x) - \frac{1}{2}\left(C_{n+1}(x) + C_0(x)\right)P_{n}(x) + B_0(x)P^{(1)}_{n-1}(x), \label{R2} \\
\Phi(x)(P^{(1)}_{n}(x))' =& \frac{1}{2}\left(C_{n+1}(x) + C_0(x)\right)P^{(1)}_{n}(x) - \gamma_{n+1}D_{n+1}(x)P^{(1)}_{n-1}(x) - D_0(x)P_{n+1}(x). \label{R3}
\end{align}
\end{proposition}

The theoretical framework established above provides all the ingredients needed for the study of differential equations satisfied by Laguerre--Hahn polynomials. In the following section, we extend this framework to develop a systematic approach for deriving homogeneous linear differential equations of any fixed order for such families.

%%%%%%%%%%%%%%%%%%%%%%%%%%%%%%%%%%%%%%%%%%%%%%%%%
%%%%%%%%%%%%%%%%%%%%%%%%%%%%%%%%%%%%%%%%%%%%%%%%%
\section{Structure relations and higher-order differential equations for Laguerre-Hahn orthogonal polynomials} \label{Section3}
%%%%%%%%%%%%%%%%%%%%%%%%%%%%%%%%%%%%%%%%%%%%%%%%%
%%%%%%%%%%%%%%%%%%%%%%%%%%%%%%%%%%%%%%%%%%%%%%%%%

In this section, we begin by constructing a sequence of structure relations, taking as starting point the fundamental structure relation \eqref{R4} given in Proposition~\ref{proposition-FSR}. This relation, which characterizes Laguerre--Hahn orthogonal polynomial sequences, will serve as the foundation for the entire construction. By successively differentiating and combining it with auxiliary relations, we obtain a hierarchy of structure relations indexed by the order of derivation. This hierarchy is the key tool that will allow us, in a second step, to derive homogeneous linear differential equations of arbitrary order for such sequences.

%%%%%%%%%%%%%%%%%%%%%%%%%%%%%%%%%%%%%%%%%%%%%%%%%%
\begin{theorem}\label{th:0.1}
If \(\{P_n\}_{n \geq 0}\) is a Laguerre-Hahn monic orthogonal polynomial sequence, then it satisfies the following structure relation for each integer $i\geq 1$: 
\begin{equation}
G_{0,i}(x;n) P_{n-1}^{(1)}(x) + G_{1,i}(x;n) P_{n}^{(1)}(x) + H_i(x;n) P_n(x) = F_i(x;n),\ n\geq 0, \label{ST_RR_i}
\end{equation}
with
\begin{equation}
F_i(x;n) = \Phi^i(x) P_{n+1}^{[i]}(x) + \sum_{k=0}^{i-1} M_{k,i}(x;n) P_{n+1}^{[k]}(x), \ n\geq 0,\label{RR_F}
\end{equation}
where $P_{n+1}^{[k]}(x)$ denotes the $k$-th derivative of $P_{n+1}(x)$ with respect to  $x$.

The coefficients $G_{0,i}(x;n)$, $G_{1,i}(x;n)$, and $H_i(x;n)$,  satisfy the following initial conditions for $n\geq 0$,
\begin{align}
G_{0,1}(x;n)=&0,  \label{G01}\\
G_{1,1}(x;n)=&B_{0}(x),  \label{G11}\\
H_{1}(x;n)=&-\gamma_{n+1}D_{n+1}(x), \label{H1}
\end{align}
and recurrence relations for $i\geq 1$, and $n\geq 0$,
\begin{eqnarray}
G_{0,i+1}(x;n) &=& -\frac12\bigl(C_{n+1}(x)-C_0(x)\bigr) G_{0,i}(x;n) - \gamma_{n+1} D_{n+1}(x) G_{1,i}(x;n) \notag\\
&&+ B_0(x) H_i(x;n) + \Phi(x) G_{0,i}'(x;n) ,\label{RR_G0}\\
G_{1,i+1}(x;n) &=& D_n(x) G_{0,i}(x;n) + \frac12\bigl(C_{n+1}(x)+C_0(x)\bigr) G_{1,i}(x;n) + \Phi G_{1,i}' (x;n),\label{RR_G1}\\
H_{i+1}(x;n) &= &-D_0(x) G_{0,i}(x;n) - \frac12\bigl(C_{n+1}(x)+C_0(x)\bigr) H_i(x;n) + \Phi H_i'(x;n) .\label{RR_H}
\end{eqnarray}
The coefficients $M_{k,i}(x;n)$, satisfy the following initial condition, for $n\geq 0$,
\begin{align}
M_{0,1}(x;n)=&-\frac{1}{2}\left(C_{n+1}(x)- C_{0}(x)\right), \label{M01}
\end{align}
and recurrence relations for $i\geq 1$, and $n\geq 0$,
\begin{align}
M_{0,i+1}(x;n) = & \Phi(x) M_{0,i}' (x;n)+ D_0(x) G_{1,i}(x;n) - D_n(x) H_i(x;n),\label{RR_M0}\\
M_{k,i+1}(x;n) = &\Phi(x) M_{k-1,i}(x;n) + \Phi(x) M_{k,i}'(x;n),\  1 \le k \le i, \label{RR_Mk}\\
M_{i,i}(x;n) = & \Phi^i(x). \label{RR_Mii}
\end{align}
\end{theorem}
%%%%%%%%%%%%%%%%%%%%%%%%%%%%%%%%%%%%%%%%%%%%%%%%%%
\begin{proof}
We proceed by induction on $i\ge 1$. For $i=1$, rewriting the structure relation \eqref{R4} as
\[
B_0(x)P_n^{(1)}(x)-\gamma_{n+1}D_{n+1}(x)P_n(x)
=\Phi(x)P_{n+1}'(x)-\tfrac12\big(C_{n+1}(x)-C_0(x)\big)P_{n+1}(x),
\]
and noting that $P_{n+1}^{[1]}=P_{n+1}'$, we recover exactly \eqref{ST_RR_i}-\eqref{RR_F} with $G_{0,1}=0$, $G_{1,1}=B_0$, $H_1=-\gamma_{n+1}D_{n+1}$, and $M_{0,1}=-\tfrac12(C_{n+1}-C_0)$, in agreement with \eqref{G01}--\eqref{M01}.\\
Assume \eqref{ST_RR_i}-\eqref{RR_F}  holds for some $i\ge1$:
\[
G_{0,i}P_{n-1}^{(1)}+G_{1,i}P_n^{(1)}+H_iP_n=F_i,
\qquad
F_i=\Phi^iP_{n+1}^{[i]}+\sum_{k=0}^{i-1}M_{k,i}P_{n+1}^{[k]}.
\]
Differentiating \eqref{ST_RR_i}-\eqref{RR_F}  and multiplying by $\Phi$ yields
\[
\Phi G_{0,i}'P_{n-1}^{(1)}+G_{0,i}\,\Phi\big(P_{n-1}^{(1)}\big)'
+\Phi G_{1,i}'P_n^{(1)}+G_{1,i}\,\Phi\big(P_n^{(1)}\big)'
+\Phi H_i'P_n+H_i\,\Phi P_n'
=\Phi F_i'.
\]
Substituting $\Phi(P_{n-1}^{(1)})'$, $\Phi(P_n^{(1)})'$, and $\Phi P_n'$ by their expressions \eqref{R1}--\eqref{R3} expresses the left-hand side, after regrouping, as a linear combination of $P_{n-1}^{(1)}$, $P_n^{(1)}$, $P_n$, and $P_{n+1}$ alone, with coefficients
\[
G_{0,i+1}=-\tfrac12(C_{n+1}-C_0)G_{0,i}-\gamma_{n+1}D_{n+1}G_{1,i}+B_0H_i+\Phi G_{0,i}',
\]
\[
G_{1,i+1}=D_nG_{0,i}+\tfrac12(C_{n+1}+C_0)G_{1,i}+\Phi G_{1,i}',
\]
\[
H_{i+1}=-D_0G_{0,i}-\tfrac12(C_{n+1}+C_0)H_i+\Phi H_i',
\]
respectively; multiplying $P_{n-1}^{(1)}$, $P_n^{(1)}$, and $P_n$, together with a residual term $(D_0G_{1,i}-D_nH_i)P_{n+1}$ arising from \eqref{R2}--\eqref{R3}, which is transferred to the right-hand side.

On the right-hand side, differentiating $F_i=\Phi^iP_{n+1}^{[i]}+\sum_{k=0}^{i-1}M_{k,i}P_{n+1}^{[k]}$ and re-indexing gives
\[
\Phi F_i'=\Phi^{i+1}\,P_{n+1}^{[i+1]}
+\sum_{k=1}^{i}\big(\Phi M_{k-1,i}+\Phi M_{k,i}'\big)P_{n+1}^{[k]}
+\Phi M_{0,i}'\,P_{n+1},
\]
where the boundary term $k=i$ is understood with $M_{i,i}=\Phi^i$.  Absorbing the residual term $(D_0G_{1,i}-D_nH_i)P_{n+1}$ into the $k=0$ coefficient produces
\[
F_{i+1}=\Phi^{i+1}P_{n+1}^{[i+1]}+\sum_{k=0}^{i}M_{k,i+1}P_{n+1}^{[k]},
\]
with
\[
M_{k,i+1}=\Phi M_{k-1,i}+\Phi M_{k,i}' \quad (1\le k\le i),
\qquad
M_{0,i+1}=\Phi M_{0,i}'+D_0G_{1,i}-D_nH_i.
\]

Altogether, this is precisely the relation \eqref{ST_RR_i}-\eqref{RR_F} for $i+1$
\[
G_{0,i+1}P_{n-1}^{(1)}+G_{1,i+1}P_n^{(1)}+H_{i+1}P_n=F_{i+1},
\]
together with the stated recurrences for $G_{0,i+1}$, $G_{1,i+1}$, $H_{i+1}$, $F_{i+1}$, and the $M_{k,i+1}$. This closes the induction, and \eqref{ST_RR_i}-\eqref{RR_F} holds for every $i\ge1$.
\end{proof}
%%%%%%%%%%%%%%%%%%%%%%%%%%%%%%%%%%%%%%%%%%%%%%%%%%

\begin{remark}
In particular, for \(i=1\), \(i=2\), \(i=3\), and \(i=4\), the four structure relations given in \cite[Theorem 4]{Article-1-NA} are recovered.
\end{remark}

Now, from this theorem, we will determine a homogeneous linear differential equation of arbitrary fixed order $N\geq 4$ for Laguerre-Hahn orthogonal polynomials. The main idea is to fix four integers $i_1$, $i_2$, $i_3$, and $i_4$, such that
\[
1\leq i_1<i_2<i_3<i_4:=N,
\]
for which we will obtain the corresponding four structure relations
\begin{align}
G_{0,i_1}(x;n)P^{(1)}_{n-1}(x)+G_{1,i_1}(x;n)P^{(1)}_{n}(x)+H_{i_1}(x;n) P_{n}(x)=F_{i_1}(x;n), \label{(4.2.5.1)S1}\\
G_{0,i_2}(x;n)P^{(1)}_{n-1}(x)
+ G_{1,i_2}(x;n)P^{(1)}_{n}(x)
+ H_{i_2}(x;n)P_{n}(x)
= F_{i_2}(x;n), \label{(4.2.5.1)S2}\\
G_{0,i_3}(x;n)P^{(1)}_{n-1}(x)+G_{1,i_3}(x;n)P^{(1)}_{n}(x)+H_{i_3}(x;n) P_{n}(x)=F_{i_3}(x;n), \label{(4.2.5.1)S3}\\
G_{0,i_4}(x;n)P^{(1)}_{n-1}(x)+G_{1,i_4}(x;n)P^{(1)}_{n}(x)+H_{i_4}(x;n) P_{n}(x)=F_{i_4}(x;n). \label{(4.2.5.1)S4}
\end{align}

These four relations will serve as the foundation for the next theorem, which provides the central result of this work.

%%%%%%%%%%%%%%%%%%%%%%%%%%%%%%%%%%%%%%%%%%%%%%%%%
\begin{theorem}\label{th:0.2}
Let $N$ be a fixed integer such that \(N \geq 4\). Let  $i_1$, $i_2$, $i_3$, and $i_4$ be four fixed integers such that 
\[
1 \leq i_1 < i_2 < i_3 < i_4:=N.
\]
If \(\{P_n\}_{n \geq 0}\) is a Laguerre-Hahn MOPS, then it satisfies the following
homogeneous linear differential equation of order \(N\) 
\begin{equation}\label{N_diff_LH}
\sum_{d=0}^{i_4} \mathcal{K}_d(x;n) P_{n+1}^{[d]}(x) = 0,
\end{equation}
where the coefficients $\mathcal{K}_d$ are given by :
\begin{align}
\mathcal{K}_{i_4} &= \Phi^{i_4} \Delta_4, \label{Ki4} \\[4pt]
\mathcal{K}_{i_3} &= M_{i_3,i_4} \Delta_4 - \Phi^{i_3} \Delta_3, \label{Ki3}  \\[4pt]
\mathcal{K}_{i_2} &= M_{i_2,i_4} \Delta_4 - M_{i_2,i_3} \Delta_3 + \Phi^{i_2} \Delta_2, \label{Ki2} \\[4pt]
\mathcal{K}_{i_1} &= M_{i_1,i_4} \Delta_4 - M_{i_1,i_3} \Delta_3 + M_{i_1,i_2} \Delta_2 - \Phi^{i_1} \Delta_1,\label{Ki1} \\[4pt]
\mathcal{K}_d &= M_{d,i_4} \Delta_4 - M_{d,i_3} \Delta_3 + M_{d,i_2} \Delta_2 - M_{d,i_1} \Delta_1 \quad \text{for} ~~ d \notin 
\{i_1,i_2,i_3,i_4\}, \label{Kid}
\end{align}
with 
\begin{eqnarray}
\Delta_4 &=& G_{0,i_1}\bigl(G_{1,i_2}H_{i_3} - G_{1,i_3}H_{i_2}\bigr) - G_{1,i_1}\bigl(G_{0,i_2}H_{i_3} - G_{0,i_3}H_{i_2}\bigr)  \label{Delta4}\\
&&+ H_{i_1}\bigl(G_{0,i_2}G_{1,i_3}- G_{0,i_3}G_{1,i_2}\bigr),\notag \\
\Delta_3 &=& G_{0,i_1}\bigl(G_{1,i_2}H_{i_4} - G_{1,i_4}H_{i_2}\bigr) - G_{1,i_1}\bigl(G_{0,i_2}H_{i_4} - G_{0,i_4}H_{i_2}\bigr)  \label{Delta3}\\
&&+ H_{i_1}\bigl(G_{0,i_2}G_{1,i_4}- G_{0,i_4}G_{1,i_2}\bigr), \notag\\
\Delta_2 &= &G_{0,i_1}\bigl(G_{1,i_3}H_{i_4} - G_{1,i_4}H_{i_3}\bigr) - G_{1,i_1}\bigl(G_{0,i_3}H_{i_4} - G_{0,i_4}H_{i_3}\bigr)  \label{Delta2}\\
&&+ H_{i_1}\bigl(G_{0,i_3}G_{1,i_4}- G_{0,i_4}G_{1,i_3}\bigr), \notag\\
\Delta_1 &=& G_{0,i_2}\bigl(G_{1,i_3}H_{i_4} - G_{1,i_4}H_{i_3}\bigr) - G_{1,i_2}\bigl(G_{0,i_3}H_{i_4} - G_{0,i_4}H_{i_3}\bigr) \label{Delta1}\\
&&+ H_{i_2}\bigl(G_{0,i_3}G_{1,i_4} - G_{0,i_4}G_{1,i_3}\bigr). \notag
\end{eqnarray}
\end{theorem}
%%%%%%%%%%%%%%%%%%%%%%%%%%%%%%%%%%%%%%%%%%%%%%%%%%
\begin{proof}
The $N$-th order differential equation is expressed in determinantal form from the system
\eqref{(4.2.5.1)S1}--\eqref{(4.2.5.1)S4}, that is,
\begin{equation}\label{BigDet}
\begin{vmatrix}
G_{0,i_1}(x;n) & G_{1,i_1}(x;n) & H_{i_1}(x;n) & F_{i_1}(x;n)\\
G_{0,i_2}(x;n) & G_{1,i_2}(x;n) & H_{i_2}(x;n) & F_{i_2}(x;n)\\
G_{0,i_3}(x;n) & G_{1,i_3}(x;n) & H_{i_3}(x;n) & F_{i_3}(x;n)\\
G_{0,i_4}(x;n) & G_{1,i_4}(x;n) & H_{i_4}(x;n) & F_{i_4}(x;n)
\end{vmatrix}=0,\quad n\geq 0.
\end{equation}

Expanding the determinant along the fourth column yields
\begin{equation}\label{StarEq}
\Delta_1(x;n)F_{i_1}(x;n)
-\Delta_2(x;n)F_{i_2}(x;n)
+\Delta_3(x;n)F_{i_3}(x;n)
-\Delta_4(x;n)F_{i_4}(x;n)=0,\ n\geq 0,
\end{equation}
where the polynomials $\Delta_k(x;n)$, for $1\leq k\leq 4$, are given by \eqref{Delta4}--\eqref{Delta1}, and depend only on the coefficients $G_{0,i_j}$, $G_{1,i_j}$, and $H_{i_j}$ ($1\le j\le 4$); in particular they do not involve $P_{n+1}$ or any of its derivatives.

We now substitute, in \eqref{StarEq}, the expressions of $F_{i_j}(x;n)$ given by Theorem~\ref{th:0.1}, namely
\begin{equation}\label{FijExpansion}
F_{i_j}(x;n)=\Phi^{i_j}(x)P_{n+1}^{[i_j]}(x)+\sum_{k=0}^{i_j-1}M_{k,i_j}(x;n)P_{n+1}^{[k]}(x),
\qquad j=1,2,3,4,
\end{equation}
with the convention $M_{d,i_j}\equiv 0$ for $d\geq i_j$, which allows every sum below to be indexed uniformly over $0\le d\le i_4$. Inserting \eqref{FijExpansion} into \eqref{StarEq} and collecting, for each fixed order $d$ with $0\le d\le i_4$, all the terms multiplying $P_{n+1}^{[d]}(x)$, we obtain
\begin{equation}\label{KdSum}
\sum_{d=0}^{i_4}K_d(x;n)\,P_{n+1}^{[d]}(x)=0,\quad n\geq 0,
\end{equation}
where $K_d$ collects, with the alternating signs inherited from \eqref{StarEq}, the contribution of order $d$ coming from each of the four expansions \eqref{FijExpansion}. Two types of contributions occur for a given $j$:

\begin{itemize}
\item if $d=i_j$, the term of order $d$ in $F_{i_j}$ is the \emph{leading} one, $\Phi^{i_j}P_{n+1}^{[i_j]}$, contributing $\pm\Phi^{i_j}\Delta_j$;
\item if $d<i_j$, the term of order $d$ in $F_{i_j}$ is the \emph{subordinate} one, $M_{d,i_j}P_{n+1}^{[d]}$, contributing $\pm M_{d,i_j}\Delta_j$;
\item if $d>i_j$, the expansion \eqref{FijExpansion} contains no term of order $d$ (equivalently $M_{d,i_j}\equiv0$), so $F_{i_j}$ contributes nothing to $K_d$.
\end{itemize}

Applying this to $d=i_4$: only $F_{i_4}$ has a term of order $i_4$ (its leading term), since $i_1,i_2,i_3<i_4$; hence
\[
K_{i_4}=\Phi^{i_4}\Delta_4.
\]

For $d=i_3$: $F_{i_3}$ contributes its leading term $\Phi^{i_3}P_{n+1}^{[i_3]}$ (with sign $+\Delta_3$ in \eqref{StarEq}), while $F_{i_4}$ contributes its subordinate term $M_{i_3,i_4}P_{n+1}^{[i_3]}$ (with sign $-\Delta_4$); $F_{i_1}$ and $F_{i_2}$ contribute nothing since $i_1,i_2<i_3$. Hence
\[
K_{i_3}=M_{i_3,i_4}\Delta_4-\Phi^{i_3}\Delta_3.
\]

For $d=i_2$: $F_{i_2}$ contributes its leading term ($\Phi^{i_2}\Delta_2$), while $F_{i_3}$ and $F_{i_4}$ contribute their respective subordinate terms of order $i_2$, namely $M_{i_2,i_3}$ and $M_{i_2,i_4}$, each carrying the sign attached to $F_{i_3}$ and $F_{i_4}$ in \eqref{StarEq}. This gives
\[
K_{i_2}=M_{i_2,i_4}\Delta_4-M_{i_2,i_3}\Delta_3+\Phi^{i_2}\Delta_2.
\]

For $d=i_1$: similarly, $F_{i_1}$ contributes its leading term ($-\Phi^{i_1}\Delta_1$), and $F_{i_2}, F_{i_3}, F_{i_4}$ each contribute a subordinate term of order $i_1$, so that
\[
K_{i_1}=M_{i_1,i_4}\Delta_4-M_{i_1,i_3}\Delta_3+M_{i_1,i_2}\Delta_2-\Phi^{i_1}\Delta_1.
\]

Finally, for any order $d\notin\{i_1,i_2,i_3,i_4\}$, none of the four $F_{i_j}$ has a leading term of order $d$; every contribution is subordinate, coming from each $F_{i_j}$ such that $i_j>d$. Hence
\[
K_d=M_{d,i_4}\Delta_4-M_{d,i_3}\Delta_3+M_{d,i_2}\Delta_2-M_{d,i_1}\Delta_1,
\]
where, consistently with the convention above, any $M_{d,i_j}$ with $i_j\le d$ is set to $0$.

Substituting these coefficients into \eqref{KdSum} yields exactly
\[
\sum_{d=0}^{i_4}K_d(x;n)\,P_{n+1}^{[d]}(x)=0,\quad n\geq 0,
\]
which is the announced differential equation~\eqref{N_diff_LH}. Since $K_{i_4}=\Phi^{i_4}\Delta_4$ multiplies the highest-order term $P_{n+1}^{[i_4]}$, this identity is a differential equation of order $N=i_4$ satisfied by $P_{n+1}$, which completes the proof.
\end{proof}
%%%%%%%%%%%%%%%%%%%%%%%%%%%%%%%%%%%%%%%%%%%%%%%%%%

\begin{remark} For the particular case $N=4$, there exists only one choice for the integers \(i_1\), $i_2$, and $i_3$ ($i_4=4$)
namely $(i_1,i_2,i_3,i_4)=(1,2,3,4)$, for which the following formulas given in \cite[Theorem 5]{Article-1-NA} are recovered
\[
\begin{aligned}
\mathcal{K}_4 &= \Phi^4 \Delta_4 = \mathcal{A}, \\
\mathcal{K}_3 &= M_{3,4} \Delta_4 - \Phi^3 \Delta_3 = \mathcal{B}, \\
\mathcal{K}_2 &= M_{2,4} \Delta_4 - M_{2,3} \Delta_3 + \Phi^2 \Delta_2 = \mathcal{C}, \\
\mathcal{K}_1 &= M_{1,4} \Delta_4 - M_{1,3} \Delta_3 + M_{1,2} \Delta_2 - \Phi \Delta_1 = \mathcal{D}, \\
\mathcal{K}_0 &= M_{0,4} \Delta_4 - M_{0,3} \Delta_3 + M_{0,2} \Delta_2 - M_{0,1} \Delta_1 = \mathcal{E}.
\end{aligned}
\]
\end{remark}

\begin{remark} 
For a fixed order \(N>4\), there exists several possibilities for the integers \(i_1\), $i_2$, and $i_3$ ($i_4=N$). 
One should begin by fixing a choice for those three integers; then, by applying the two preceding theorems, we get a specific differential equation of order \(N\).

For example, for \(N=5\) (\(i_4=5\)), there exists \(4\) possibilities for \((i_1,i_2,i_3)\) :
\[
(i_1,i_2,i_3)\in 
\Big\{
(1,2,3),\;
(1,2,4),\;
(1,3,4),\;
(2,3,4)
\Big\}.
\]

For each choice, one gets a different differential equation of order \(N=5\), as is illustrated in section \ref{Section5} for several examples.
\end{remark} 

%%%%%%%%%%%%%%%%%%%%%%%%%%%%%%%%%%%%%%%%%%%%%%%%%
\subsection{The semiclassical case} \label{Section3.1}
%%%%%%%%%%%%%%%%%%%%%%%%%%%%%%%%%%%%%%%%%%%%%%%%%

Theorem \ref{th:0.2} provides a general framework for deriving differential equations satisfied by any strict Laguerre-Hahn polynomial sequence. 
In what follows, we examine the semiclassical situation, which includes the classical case as a particular instance. 
In fact, in the semiclassical case, i.e., when $B_0(x) = 0$, all the coefficients $G_{i,j}$  in Theorem~\ref{th:0.1} vanish.
Consequently, the entries in the first two columns of the determinant \eqref{BigDet} are identically zero, and the resulting differential equation in Theorem \ref{th:0.2} reduces to the trivial identity $0=0$. 
Therefore, one must substitute $B_0(x)=0$ directly into the structure relations of Theorem 
\ref{th:0.1} (see Corollary \ref{Cor_STR_SC}), and deduce the corresponding differential equations from these reduced relations (see Proposition \ref{Cor_DEij_SC}). 

%%%%%%%%%%%%%%%%%%%%%%%%%%%%%%%%%%%%%%%%%%%%%%%%%
\begin{corollary}\label{Cor_STR_SC}
If \(\{P_n\}_{n \geq 0}\) is a semiclassical monic orthogonal polynomial sequence, then the coefficients of the structure relation \eqref{ST_RR_i} satisfy the following initial conditions and recurrence relations for $i\geq 1$ and $n\geq 0$,
\begin{eqnarray}
G_{0,i}(x;n)& =&0,   \label{SC-G0i}\\
G_{1,i}(x;n) & =&0,  \label{SC-G1i}\\
H_{1}(x;n)&=&-\gamma_{n+1}D_{n+1}(x), \label{SC-H1}\\
H_{i+1}(x;n) &= &- \frac12\bigl(C_{n+1}(x)+C_0(x)\bigr) H_i (x;n)+ \Phi(x) H_i' (x;n),\label{SC-RR_H}\\
M_{0,1}(x;n) &=&-\frac{1}{2}\left(C_{n+1}(x)- C_{0}(x)\right),  \label{SC-M01}\\
M_{0,i+1}(x;n) & = & \Phi(x) M_{0,i}'(x;n)- D_n(x) H_i(x;n),\label{SC-RR_M0}\\
M_{k,i+1}(x;n) &= &\Phi(x) M_{k-1,i}(x;n) + \Phi(x) M_{k,i}'(x;n),\  1 \le k \le i. \label{SC-RR_Mk}
\end{eqnarray}

Consequently, the structure relation of index $i\geq 1$ takes the form
\begin{equation}\label{SRi_SC}
H_i(x;n) P_n(x) = \Phi^i(x)P_{n+1}^{[i]}(x) +\sum_{k=0}^{i-1} M_{k,i}(x;n) P_{n+1}^{[k]}(x), \ n\geq 0. 
\end{equation}
%%%%%%%%%%%%%%%%%%%%%%%%%%%%%%%%%%%%%%%%%%%%%%%%%
\end{corollary}
\begin{proof}
In the semiclassical case, $B_0(x)=0$. Hence, by \eqref{G11}, we obtain \eqref{SC-G1i} for $i=1$.
Now assume that \eqref{SC-G0i} and \eqref{SC-G1i} hold for some fixed $i$. Then, by the recurrence relations \eqref{RR_G0} and \eqref{RR_G1}, it follows that \eqref{SC-G0i} and \eqref{SC-G1i} also hold for $i+1$. This completes the inductive proof of \eqref{SC-G0i}-\eqref{SC-G1i}.
The identities \eqref{SC-H1}-\eqref{SC-RR_Mk} follow by applying \eqref{SC-G0i} and \eqref{SC-G1i} to the identities
\eqref{H1}, and  \eqref{RR_H}-\eqref{RR_Mk} from Theorem \ref{th:0.1}.
\end{proof}
%%%%%%%%%%%%%%%%%%%%%%%%%%%%%%%%%%%%%%%%%%%%%%%%%

For each $i\geq 2$, the next proposition establishes $i-1$ distinct linear differential equations of order $i$ by combining the structure relations corresponding to the indices $i$ and $j$, where $1\leq j<i$.

%%%%%%%%%%%%%%%%%%%%%%%%%%%%%%%%%%%%%%%%%%%%%%%%%
\begin{proposition}\label{Cor_DEij_SC}
If \(\{P_n\}_{n \geq 0}\) is a semiclassical monic orthogonal polynomial sequence, then it satisfies  the following  homogeneous linear differential equation of order $i\geq 2$, indexed by $j$,  where $1\leq j<i$:
\begin{equation}\label{i_diff_SC}
\sum_{d=0}^{i} \mathcal{K}_d(x;n) P_{n+1}^{[d]}(x) = 0,\quad  n\geq 0,
\end{equation}
where
\begin{eqnarray}
\mathcal{K}_i(x;n) & = & H_j(x;n)\Phi^i(x), \ d=i,\label{SC_Ki}\\
\mathcal{K}_d(x;n) & = & H_j(x;n)M_{d,i}(x;n), \ j+1\leq d \leq i-1, \label{SC_Kdj1i1}\\
\mathcal{K}_j(x;n) & = & H_j(x;n)M_{j,i}(x;n)-H_i(x;n)\Phi^j(x), \ d=j, \label{SC_Kdj}\\
\mathcal{K}_d(x;n) & = & H_j(x;n)M_{d,i}(x;n)-H_i(x;n)M_{d,j}(x;n), \ 0\leq d \leq j-1, \label{SC_Kd0j1}
\end{eqnarray}
with the coefficients $H_l(x;n)$ and $M_{k,l}(x;n)$ defined in Theorem \ref{th:0.1} under the condition $B_0(x)=0$.
%%%%%%%%%%%%%%%%%%%%%%%%%%%%%%%%%%%%%%%%%%%%%%%%%
\end{proposition}
\begin{proof}
From the Corollary \ref{Cor_STR_SC}, we know that $G_{0,l}=0$ and $G_{1,l}=0$, which simplify the structure relations in Theorem \ref{th:0.1}.
Setting $i=j$ in \eqref{ST_RR_i} together with \eqref{RR_F} in that theorem, we obtain the following structure relation of index $j$
$$H_j(x;n) P_n(x) =\Phi^j(x) P_{n+1}^{[j]}(x) + \sum_{k=0}^{j-1} M_{k,j}(x;n) P_{n+1}^{[k]}(x),\ n\geq 0.$$
Multiplying both sides of the structure relation \eqref{ST_RR_i}-\eqref{RR_F} of index $i$ by $H_j(x;n)$ and substituting the above identity,  we obtain
\begin{eqnarray}
&& H_i(x;n) \left(\Phi^j(x) P_{n+1}^{[j]}(x) + \sum_{k=0}^{j-1} M_{k,j}(x;n) P_{n+1}^{[k]}(x)\right)= \notag\\
 &&H_j(x;n) \Phi^i(x) P_{n+1}^{[i]}(x) + H_j(x;n) \sum_{k=0}^{i-1} M_{k,i}(x;n) P_{n+1}^{[k]}(x), \ n\geq 0. \notag
\end{eqnarray}
Rearranging terms yields
\begin{eqnarray}
&&H_j(x;n) \Phi^i(x) P_{n+1}^{[i]}(x) + \sum_{k=j+1}^{i-1} H_j(x;n)M_{k,i}(x;n) P_{n+1}^{[k]}(x)\notag\\
&&+\left(H_j(x;n) M_{j,i}(x;n) -H_i(x;n)\Phi^j(x)\right) P_{n+1}^{[j]}(x)\notag\\
&&+ \sum_{k=0}^{j-1}\left(H_j(x;n)M_{k,i}(x;n)-H_i(x;n)M_{k,j}(x;n)\right)P_{n+1}^{[k]}(x)=0, \ n\geq 0. \notag
\end{eqnarray}
Comparing this equation with \eqref{i_diff_SC}, we obtain the desired expressions for the coefficients $\mathcal{K}_d(x;n)$.
\end{proof}

%%%%%%%%%%%%%%%%%%%%%%%%%%%%%%%%%%%%%%%%%%%%%%%%%%%%%%%%%
%%%%%%%%%%%%%%%%%%%%%%%%%%%%%%%%%%%%%%%%%%%%%%%%%%%%%%%%%
%%%%%%%%%%%%%%%%%%%%%%%%%%%%%%%%%%%%
\section{Algorithm for the symbolic computation of structure relations and differential equations}\label{Section4}
%%%%%%%%%%%%%%%%%%%%%%%%%%%%%%%%%%%%

The theoretical developments presented thus far provide a general and constructive method for deriving structure relations and differential equations satisfied by strict Laguerre–Hahn, semiclassical, and classical orthogonal polynomial sequences. This method is implemented in the following algorithm, named {\it HoDELH.nb}, developed in {\it Mathematica$^{\circledR}$}.

We emphasize that the coefficients of the structure relations do not admit common factors, since each structure relation is normalized by a coefficient equal to a power of the monic polynomial $\Phi(x)$. In contrast, the coefficients of the differential equations generally admit common factors. In such cases, the implementation factors them out and returns the corresponding reduced coefficients.

%%%%%%%%%%%%%%%%%%%%%%%%%%%%%%%%%%%%%%%%%%%%%%%%%%

\vspace{0.5cm}

\noindent {\bf Algorithm {\it HoDELH}} ({\it Higher-order Differential Equations for Laguerre-Hahn})

\vspace{0.25cm}

\begin{enumerate} 

\item {\bf Input Data}

\vspace{0.25cm}

- For Laguerre-Hahn sequences,

\vspace{0.25cm}

1.1. Coefficients of the recurrence relation \eqref{ic_TTRR}-\eqref{TTRR}: $\beta_{n}$, $\gamma_{n+1}$, $n\geq 0$.

\vspace{0.15cm}

1.2. Coefficients of the Stieltjes equation \eqref{Riccati}: $\Phi(z)$, $B(z)$, $C(z)$, and $D(z)$.

\vspace{0.15cm}

1.3. Coefficients of the structure relation \eqref{R4}: 

$B_0(x)=B(x)$,  $C_0(x)=C(x)$, $D_0(x)=D(x)$, $C_{n+1}(x)$, 
$D_{n+1}(x)$, $n\geq 0$.

\vspace{0.25cm}

- For strict Laguerre-Hahn sequences:

\vspace{0.15cm}

1.4.  The order $N\geq 4$ of the differential equation.

\vspace{0.15cm}

1.5. Three integers $i_1$, $i_2$, $i_3$, such that  $1\leq i_1 < i_2 < i_3 < i_4:=N$.

\vspace{0.25cm}

- For semiclassical and classical sequences:

\vspace{0.15cm}

1.6. An integer $l\geq 1$.

\vspace{0.15cm}

1.7. Two integers $i$, and $j$, such that $1\leq j < i-1$, $i\geq 2$.

\vspace{0.25cm}

\item {\bf Computation of structure relations for Laguerre-Hahn orthogonal polynomial sequences}

\vspace{0.25cm}

Compute the coefficients of the structure relation of index $i$, given by \eqref{ST_RR_i}--\eqref{RR_F} in Theorem~\ref{th:0.1}, from the input data in Items 1.1--1.3 by using the initial conditions and the recurrence relations \eqref{G01}--\eqref{RR_Mii}.

\vspace{0.25cm}

\item {\bf Computation of differential equations for strict Laguerre-Hahn orthogonal polynomial sequences}

\vspace{0.25cm}

3.1. Compute the coefficients of the structure relations of indices $i_1$, $i_2$, $i_3$, and $i_4=N$ using Step 2.

\vspace{0.15cm}

3.2. Compute the coefficients of the $N$th-order differential equation \eqref{N_diff_LH} in Theorem~\ref{th:0.2}, corresponding to the indices $i_1$, $i_2$, $i_3$, and $i_4=N$, from the coefficients obtained in Step 3.1 by using the initial conditions and the recurrence relations \eqref{Ki4}--\eqref{Kid} together with \eqref{Delta4}--\eqref{Delta1}.

\vspace{0.25cm}

% 4
\item {\bf Computation of structure relations for semiclassical and classical orthogonal polynomial sequences}

\vspace{0.25cm}

Compute the coefficients of the structure relation of index $l$, $l\ge1$, from the input data in Items 1.1--1.3 by using the initial conditions and the recurrence relations \eqref{SC-G0i}--\eqref{SC-RR_Mk} in Corollary~\ref{Cor_STR_SC}. Alternatively, these coefficients may be obtained from the general procedure described in Step 2.
 
\vspace{0.25cm}

\item {\bf Computation of differential equations for semiclassical and classical orthogonal polynomial sequences} 

\vspace{0.25cm} 

5.1. Compute the coefficients of the structure relations of indices $i$ and $j$ using Step 4 or, alternatively, the general procedure in Step 2.

\vspace{0.15cm}

5.2. Compute the coefficients of the differential equation of order $i$, corresponding to the index $j$, from the coefficients of the structure relations of indices $i$ and $j$ obtained in Step 5.1 by using relations \eqref{SC_Ki}--\eqref{SC_Kd0j1} in Proposition~\ref{Cor_DEij_SC}.

\vspace{0.25cm}

\item {\bf Computation of reduced coefficients of differential equations}

\vspace{0.25cm}

Compute the greatest common divisor (GCD) of the coefficients of each differential equation and, whenever permitted by the regularity conditions, divide all coefficients by this GCD to obtain the corresponding reduced coefficients.

\vspace{0.25cm}

\item {\bf Computation of orthogonal polynomials}

\vspace{0.25cm}

Compute the polynomials $P_n(x)$, and $P^{(1)}_n(x)$, for prescribed values of $n$, using the recurrence relations \eqref{ic_TTRR}--\eqref{TTRR} and \eqref{ic_ASSTTRR}--\eqref{ASSTTRR}, respectively.

\vspace{0.25cm}

\item {\bf Presentation of results}

\vspace{0.25cm}

Express the coefficients of the structure relations and differential equations in the canonical basis $\{x^i\}_{i\geq 0}$, with factorization carried out, whenever possible, with respect to $n$ and the parameters of the sequence.

\end{enumerate}

%%%%%%%%%%%%%%%%%%%%%%%%%%%%%%%%%%%%%%%%%%%%%%%%%%%%%%%%%%%%%%
\section{Results for Laguerre-Hahn of class~0 families analogous to Hermite} \label{Section5}
%%%%%%%%%%%%%%%%%%%%%%%%%%%%%%%%%%%%%%%%%%%%%%%%%%%%%%%%%%%%%%
In this section, we investigate the structure relations and differential equations for the two cases of class-zero Laguerre--Hahn orthogonal polynomial sequences analogous to the Hermite sequence. The classical Hermite sequence is recovered in the first case. All symbolic results are obtained by implementing the algorithm {\it HoDELH} in {\it Mathematica$^{\circledR}$}. 
The characteristic elements of these sequences, which constitute the input data for the algorithm, are given in 
\cite{Bouakkaz-Maroni-1991,Article-2-Soummi}. All sequence parameters are treated symbolically.

In Case 1, we compute the coefficients of the structure relation of order 5 and the coefficients of one differential equation of order 5 also. These results illustrate the symbolic capabilities of the algorithm.

For the classical Hermite sequence and for Case 2, we deduce in Propositions \ref{Prop_SR_CLHermite} and \ref{Prop_SR_Hermite2} the initial conditions and recurrence relations satisfied by the coefficients of the structure relations of these sequences from the general results established in  Corollary \ref{Cor_STR_SC} and Theorem \ref{th:0.1}. In addition, these propositions provide closed formulas, valid for every index $i$, for several coefficients of the structure relations. The remaining coefficients are computed symbolically using the algorithm {\it HoDELH} for the first values of $i$. More precisely, for the classical Hermite sequence, we present the structure relations of indices $1\leq i\leq 10$, whereas for Case 2, we present those of indices $1\leq i\leq 9$.
 
 With respect to differential equations, for the classical Hermite sequence, Proposition  \ref{Pro_DEi_CLHermite}  provides the coefficients corresponding to $j=1$. Using the algorithm {\it HoDELH}, we compute the coefficients of the differential equations of orders $i=2,3,4,5$, and $9$. For each order $i$, we determine all differential equations corresponding to the index $j$, where $1\le j<i$. 
 For Case 2, we compute one differential equation of order five and another of order ten.

%%%%%%%%%%%%%%%%%%%%%%%%%%%%%%%%%%%%%%
\subsection{Case~1 analogous to Hermite}\label{Section5_H_Case1}
%%%%%%%%%%%%%%%%%%%%%%%%%%%%%%%%%%%%%%

The regularity conditions for this sequence are
\[
\tau,\ \lambda,\ \rho \in\mathbb{C}, \quad \rho\neq 0, \quad \tau\not=-n, \quad n\geq 1.
\]
The canonical form $u_0$ of this sequence is related to the classical Hermite form ${\cal H}$ by \cite[Proposition 4.1]{Mohamed-Imed-2025} 
$$u_{0}={\cal H}^{(\tau)}\left(\lambda;\,{0\atop\rho}\,;1\right).$$
The input data for the algorithm {\it HoDELH} are as follows:
 
\noindent  - the recurrence coefficients, 
\begin{eqnarray}
&&\beta_{0}=\lambda\ ,\ \beta_{n+1}=0,\ n\ge 0\quad ; \quad 
\gamma_{1}=\rho{{\tau+1}\over{2}},\ \gamma_{n+1}={{n+\tau+1}\over{2}},\ n\ge 1. \notag
\end{eqnarray}
- the coefficients of the functional equation, 
\begin{eqnarray}
&&\Phi(x)=1,\quad \psi(x)=2{{2-\rho}\over{\rho}}x-{{4\lambda}\over{\rho}} ,\notag\\ 
&&B(x)=2{{\rho-1}\over{\rho}}x^{2}+2\lambda{{2-\rho}\over{\rho}}x+1-\rho(\tau+1)-{{2\lambda^{2}}\over{\rho}}.\notag
\end{eqnarray}
- the coefficients of the Laguerre-Hahn structure relation,
\begin{eqnarray}
&&C_{0}(x)=2{{\rho-2}\over{\rho}}x+{{4\lambda}\over{\rho}},\quad D_{0}(x)=-{{2}\over{\rho}}, \quad
C_{n+1}(x)=-2x,\quad D_{n+1}(x)=-2,\ n\ge 0.
\notag
\end{eqnarray}

In this case, some coefficients appearing in the structure relations and the differential equations satisfy initial conditions, which can be expressed compactly using the notation
$$
\epsilon_n=1-(1-\rho)\delta_{n,0},\quad n\geq 0.
$$
However, the reduced coefficients of some of the differential equations, obtained after eliminating common factors, are free of initial conditions.

%%%%%%%%%%%%%%%%%%%%%%%%%%%%%%%%%%%%%%%%%%%%%%
\subsubsection{Structure relations}
%%%%%%%%%%%%%%%%%%%%%%%%%%%%%%%%%%%%%%%%%%%%%%

Next, we list the coefficients of the structure relation of order five obtained by applying the algorithm {\it HoDELH} using symbolic computations. 

\vspace{0.5cm}

\noindent {\bf Structure relation of order $i=5$}

%%%%%%%%%%%%%%%%%%%%%%%%%%%%%%%%%%
\begin{eqnarray}
G_{0,5}(n;x) & = &\frac{4\epsilon_n (n+\tau +1)}{\rho ^4}\Big( 
8 (\rho -1)^2 \left(\rho ^2+1\right) x^5-8 \lambda  (\rho -1)
   \left(\rho ^3-3 \rho ^2+3 \rho -5\right) x^4\notag\\
   %%%%%
&&   -4x^3 \left(\rho ^5 (\tau +1)+\rho ^4 (3 \tau -8)+\rho ^3 \left(4 \lambda ^2-7 \tau +15\right)-3 \rho ^2 \left(4 \lambda ^2-\tau
   +7\right)\right.\notag\\
   &&\left.+\left(24 \lambda ^2+13\right) \rho -20 \lambda ^2+n \left(2 \rho ^4-5 \rho
   ^3+3 \rho ^2\right) \right)\notag\\
   %%%%%
   &&-2 \lambda 
   x^2 \left(\rho ^4 (7-6 \tau )+7 \rho
   ^3 (4 \tau -5)+\rho ^2 \left(8 \lambda ^2-18 \tau +89\right)\right.\notag\\
   &&\left.-2 \left(16 \lambda ^2+39\right) \rho +40
   \lambda ^2+n \left(-4 \rho ^4+20 \rho ^3-18 \rho ^2\right) \right)\notag\\
%%%%%
   &&-2 x \left(-\rho ^5 (\tau +1) (4 \tau+1)+\rho ^4 \left(4 \tau ^2+24 \tau +1\right)\right.\notag\\
   &&+\rho ^3 \left(-14 \lambda ^2 \tau +5 \lambda ^2-13 \tau +24\right)+2 \rho^2 \left(9 \lambda ^2 \tau -26 \lambda ^2-12\right)\notag\\
   &&+2 \lambda ^2 \left(4 \lambda ^2+39\right) \rho-20 \lambda ^4 \notag\\
   &&\left.+n \left(-10 \left(\lambda ^2+1\right) \rho ^3+18 \lambda ^2 \rho
   ^2-2 \rho ^5 (\tau +1)+3 \rho ^4 (\tau +4)\right)\right)\notag\\
   &&
-\lambda  \left(-\rho ^4 (2
   \tau +1) (4 \tau +11)+\rho ^3 (26 \tau -25)-2 \rho ^2 \left(6 \lambda ^2 \tau -5 \lambda ^2-24\right)\right.\notag\\
&&\left.
   +n \left(-12 \lambda ^2 \rho ^2+\rho ^4 (-6 \tau -13)+20 \rho ^3\right)+8 \lambda ^4-52 \lambda ^2 \rho\right)\Big) ,\notag\\
%\end{eqnarray}
%%%%%%%%%%%%%%%%%%%%%%%%%%%%%%%%%%
%\begin{eqnarray}
G_{1,5}(n;x) & = &\frac{4}{\rho^5} \left(
8 (\rho -1) x^6-8 \lambda  (5 \rho -6) x^5\right.\notag\\
&&+4 \left(-2 \left(\rho ^2-3 \rho +5\right) \rho ^3
   (n+\tau +1)+\rho ^2 (6 n+5 \tau -9)\right.\notag\\
   &&\left.+5 \left(4 \lambda ^2+3\right) \rho -30 \lambda ^2\right)
   x^4\notag\\
   %%%%%%%%%%%%%%%%%%%%%%%
   &&-8 \lambda  \left(\rho ^5 (-n-\tau -1)+3 \rho ^3 (2 \rho -5) (n+\tau +1)+\rho ^2 (12 n+10
   \tau -13)\right.\notag\\
   &&\left.+10 \left(\lambda ^2+3\right) \rho -20 \lambda ^2\right) x^3\notag\\
      %%%%%%%%%%%%%%%%%%%%%%%
   &&+2 \left(2 \rho ^4 (n+\tau +1) \left(\rho ^2 (\tau +1)+(n-10) \rho+6 \lambda ^2 -n+\tau +30\right)\right.\notag\\
   &&\left.
   -3 \rho ^3 \left(20 \lambda ^2 (n+\tau +1)+14 n+12 \tau -11\right)+3 \rho ^2 \left(4 \lambda ^2 (6 n+5 \tau -4)-15\right)  \right.\notag\\
   &&\left.
   +20 \left(\lambda ^2+9\right) \lambda ^2 \rho  -60 \lambda ^4  \right) x^2\notag\\
      %%%%%%%%%%%%%%%%%%%%%%%
   &&-2 \lambda  \left(\rho ^4 (n+\tau +1) ((2 n-11) \rho -2 (2 n-2 \tau -41))\right.\notag\\
   &&
   +\rho ^3 \left(-20 \lambda ^2 (n+\tau +1)-3 (28 n+24 \tau +3)\right)\notag\\
   && \left. +2 \rho ^2 \left(2 \lambda ^2 (12 n+10 \tau -3)-45\right)
    -24 \lambda ^4+4 \left(\lambda ^2+30\right) \lambda ^2 \rho\right) x\notag\\
   %%%%%%%%%%%%%%%%%%%%%%%%%%%%%%%%%%%%
   &&-2 \rho ^5 (n+\tau +1) \left(\rho  (\tau +1) (n+2 \tau )-n+11\right)\notag\\
   &&
   +\rho ^4 \left(-4 \lambda ^2 (n-\tau -11) (n+\tau +1)+22 n+19 \tau +7\right)\notag\\
   &&-3 \rho ^3
   \left(4 \lambda ^2 (7 n+6 \tau +2)-5\right)+2 \lambda ^2 \rho ^2
   \left(2 \lambda ^2 (6 n+5 \tau +1)-45\right) \notag\\
   &&\left.+60 \lambda ^4 \rho  -8 \lambda ^6
\right),\notag\\
%%%%%%%%%%%%%%%%%%%%%%%%%%%%%%%%%%
H_5(n;x) & = & \frac{4 (n+\tau +1)}{\rho ^4}\left(4 \left(2 \rho ^3-2 \rho ^2+2 \rho -1\right) x^4-8 \lambda 
   (\rho -1) \left(\rho ^2-\rho +2\right) x^3\right.\notag\\
   &&
   -4 x^2 \left(\rho ^4 (\tau +1) +\rho ^3 (n+2 \tau -8)+\rho ^2 \left(2 \lambda ^2-n-\tau
   +10\right)+\right.\notag\\
   &&\left. -6 \left(\lambda ^2+1\right) \rho +6 \lambda ^2\right)\notag\\
   &&
   -2 \lambda  x \left(\rho ^3 (-2 n-4 \tau +7)+4 \rho ^2 (n+\tau
   -6)+ 4 \left(\lambda ^2+6\right) \rho -8 \lambda ^2\right)\notag\\
   &&
+2 \rho ^4 (\tau +1) (n+2 \tau )-2 \rho ^3 (n+2 \tau -9)+\rho ^2 \left(4 \lambda ^2 (n+\tau -2)-15\right) \notag\\
   &&\left. +24 \lambda ^2 \rho  -4 \lambda ^4
\right),\notag\\
%%%%%%%%%%%%%%%%%%%%%%%%%%%%%%%%%%
M_{4,5}(n;x) & = & \frac{2}{\rho}\left((\rho -1) x+\lambda \right) ,\notag\\
M_{3,5}(n;x) & = & -\frac{2}{\rho ^2}\left(2 (\rho +1) x^2-2 \lambda 
   (\rho -2) x+\rho ^2 (-n-2 \tau -6)+5 \rho -2 \lambda ^2\right),\notag\\
   %%%%%%%%%%%%%%%%%%%%%%
M_{2,5}(n;x) & = & \frac{4}{\rho ^3}\left(2 (\rho -1) x^3-2 \lambda  (2 \rho
   -3) x^2+x \left((n-8) \rho ^2 +\left(2 \lambda ^2+9\right) \rho -6 \lambda ^2\right) \right.\notag\\
   &&\left.+\lambda  \left((4-n) \rho ^2-9 \rho+2 \lambda ^2\right) \right),\notag\\
   %%%%%%%%%%%%%%%%%%%%%%
M_{1,5}(n;x) & = &-\frac{4}{\rho ^4}\left( 4 (\rho -1) x^4-4 \lambda  (3 \rho -4) x^3\right.\notag\\
&&
   +2 x^2 \left(-4 \rho ^3 (n+\tau +1)+\rho ^2 (3 n+2 \tau
   -9)+6 \left(\lambda ^2+2\right) \rho -12 \lambda ^2 \right)\notag\\
   &&-2 \lambda  x \left(-4 \rho ^3 (n+\tau +1)+\rho ^2 (6 n+4 \tau
   -11) +2 \left(\lambda ^2+12\right) \rho -8 \lambda ^2\right)\notag\\
   &&
 +4 \rho ^4 (\tau +1) (n+\tau +1)+\rho ^3 (-7 n-4 \tau +8)\notag\\
 &&\left.+\rho ^2 \left(4 \lambda ^2 \tau -4 \lambda ^2+6 \lambda
   ^2 n-15\right)  -4 \lambda ^4+24 \lambda ^2 \rho \right),\notag\\
   %%%%%%%%%%%%%%%%%%%%%%%%%
M_{0,5}(n;x) & = &\frac{8}{\rho ^5}\left( 4 (\rho -1) x^5-4 \lambda  (4 \rho -5) x^4\right.\notag\\
&&
   +2 x^3 \left(4 \rho ^4 (n+\tau +1)-8 \rho ^3 (n+\tau
   +1)+\rho ^2 (5 n+4 \tau -9) \right.\notag\\
   &&\left.+2 \left(6 \lambda ^2+7\right) \rho -20 \lambda ^2 \right)\notag\\
   &&
   -2 \lambda  x^2
   \left(4 \rho ^4 (n+\tau +1)-16 \rho ^3
   (n+\tau +1)+3 \rho ^2 (5 n+4 \tau -6) \right.\notag\\
   &&\left.+2 \left(4 \lambda ^2+21\right) \rho -20 \lambda ^2
   \right)\notag\\
   &&
   +x \left(-4 \rho ^5 (\tau +1) (n+\tau +1)+4 \rho
   ^4 (\tau +9) (n+\tau +1) \right.\notag\\
   &&+\rho ^3 \left(-16 \lambda ^2 (n+\tau +1)-27 n-22 \tau +6\right)+3 \rho ^2 \left(2 \lambda ^2 (5 n+4 \tau -3)-11\right)\notag\\
   &&\left.-20 \lambda ^4+4 \lambda ^2 \left(\lambda ^2+21\right) \rho \right)  \notag\\
   &&
   +\lambda  \left(-2 \rho ^4 (2 \tau +9) (n+\tau +1) +\rho ^3 (27 n+22 \tau +2) \right.\notag\\
   &&\left.\left.+\rho ^2 \left(-8 \lambda ^2 \tau -10 \lambda ^2
   n+33\right)-28 \lambda ^2 \rho +4 \lambda ^4\right) \right)
.\notag
\end{eqnarray}

%%%%%%%%%%%%%%%%%%%%%%%%%%%%%%%%%%%%%%%%%%%%%%
\subsubsection{Differential equations}
%%%%%%%%%%%%%%%%%%%%%%%%%%%%%%%%%%%%%%%%%%%%%%

Next, we list the reduced coefficients of the differential equation of order $N=5$ corresponding to $(i_1,i_2,i_3,i_4)=(1,2,3,5)$
for $n\geq 0$, obtained by applying the {\it HoDELH} using symbolic computations. 

The greatest common factor between the coefficients ${\cal K}_i(n;x)$, for $i=0,2,3,4,5$, are
$$c(0;x)=\frac{4 (\tau +1)^2}{\rho },\quad c(n;x)=\frac{4 (n+\tau +1)^2}{\rho ^3}, \ n\geq 1.$$
\begin{eqnarray}
\widehat{{\cal K}}_{5}(n;x) & = & -\rho \Big( 8 (n+1) (\rho -1)^2 x^4 -4 \lambda  (\rho -1) x^3 (4 n( \rho -2)+3 \rho -8)\notag\\
   %%%%%%%%%%%%%%%%%%%%%
   &&+4 x^2 \left(-\rho ^3 (\tau +1)+\rho ^2 \left(\lambda ^2-\tau +2\right)+\rho  \left(-10 \lambda ^2+2 \tau -1\right)\right.\notag\\
   &&\left.+n \left(2 \rho ^2 \left(\lambda ^2+\tau +2\right)-2 \left(6 \lambda
   ^2+1\right) \rho +12 \lambda ^2-2 \rho ^3 (\tau +1)\right)+12
   \lambda ^2 \right)\notag\\
    %%%%%%%%%%%%%%%%%%%%%
   && +2 \lambda  x \left(\rho^3 (\tau +1)-5 \rho ^2+2 \rho  \left(3 \lambda ^2-4 \tau +2\right)\right.\notag\\
   &&\left.+n \left(8 \left(\lambda^2+1\right) \rho -16 \lambda ^2+4 \rho ^3 (\tau +1)-4 \rho ^2 (2 \tau +3)\right)-16 \lambda ^2 \right)\notag\\
   %%%%%%%%%%%%%%%%%%%%%%%%%%%%%%%%%%%%%%%%%
   &&+2 \rho ^3 (\tau +1) (2 \tau +3)+\rho ^2 \left( \lambda ^2(4 \tau +7 )-10 \tau -12\right)+2 \rho 
   \left(2 \lambda ^2 (2\tau -1)+3\right)\notag\\
   &&+n \left(2 \rho ^4 (\tau +1)^2-4 \rho
   ^3 (\tau +1)+2 \rho ^2 \left(4\lambda ^2 (\tau +1)+1\right)-8 \lambda ^2 \rho +8 \lambda ^4\right)+8 \lambda ^4\Big),\notag\\
%%%%%%%%%%%%%%%%%%%%%%%%%%%%%%%%%%%%%%%%%%%%%%%%%%%%%%%%%%
\widehat{{\cal K}}_{4}(n;x)  & = &-2 (x(\rho  -1)+\lambda ) \Big(8 (n+1) (\rho -1)^2 x^4
-4 \lambda  (\rho -1) x^3 (4 n (\rho -2)+3 \rho -8)\notag\\
%%%%%%%%%%%%%%%%%%%%%%%%%%%%%
   &&+4 x^2 \left(-\rho ^3 (\tau +1)+\rho ^2 \left(\lambda ^2-\tau +2\right)+\rho  \left(-10 \lambda ^2+2 \tau -1\right)+12
   \lambda ^2\right.\notag\\
   &&\left.+n \left(-2 \rho ^3 (\tau +1)+2 \rho ^2 \left(\lambda ^2+\tau +2\right)-2 \left(6 \lambda
   ^2+1\right) \rho +12 \lambda ^2\right) \right)\notag\\
   %%%%%%%%%%%%%%%%%%%%%%%%%%%%%
   &&+2 \lambda  x \left(\rho
   ^3 (\tau +1)-5 \rho ^2+2 \rho  \left(3 \lambda ^2-4 \tau +2\right)\right.\notag\\
   &&\left.+n \left(8 \left(\lambda
   ^2+1\right) \rho -16 \lambda ^2+4 \rho ^3 (\tau +1)-4 \rho ^2 (2 \tau +3)\right)-16 \lambda ^2\right)\notag\\
   %%%%%%%%%%%%%%%%%%%%%%%%%%%%%
   &&+2 \rho ^3 (\tau +1) (2 \tau +3)+\rho ^2 \left( \lambda ^2(4 \tau +7)-10 \tau -12\right)+2 \rho 
   \left(2 \lambda ^2 (2\tau -1)+3\right)\notag\\
   &&+n \left(2 \rho ^4 (\tau +1)^2-4 \rho
   ^3 (\tau +1)+2 \rho ^2 \left(4
   \lambda ^2 \tau +4 \lambda ^2+1\right)-8 \lambda ^2 \rho +8 \lambda ^4\right)+8 \lambda ^4
\Big),\notag\\
%%%%%%%%%%%%%%%%%%%%%%%%%%%%%%%%%%%%%%%%%%%%%%%%%%%%%%%%%
\widehat{{\cal K}}_{3}(n;x) & = & -2\Big(
-16 (n+1) (\rho -1)^2 \rho  x^6+8 \lambda  (\rho -1) \rho  (4  n( \rho-2)+3 \rho -8)
   x^5\notag\\
   %%%%%%%%%%%%%%%%%%
   &&+8 x^4\left(\rho ^4 (\tau +1)+\rho ^3 \left(-\lambda ^2+5 \tau -2\right)+5 \rho ^2 \left(2 \lambda ^2-2 \tau
   +1\right)\right.\notag\\
   &&-4 \rho  \left(3 \lambda ^2-\tau +2\right)+n^2 \left(2 \rho ^3-4 \rho ^2+2 \rho \right)+4\notag\\
   &&+n \left(2 \rho ^4
   (\tau +1)-2 \rho ^3 \left(\lambda ^2-\tau +1\right)+2 \rho ^2 \left(6
   \lambda ^2-4 \tau +1\right)\right.\notag\\
   &&\left.\left.-2 \rho  \left(6 \lambda ^2-2 \tau +3\right)+4\right)\right) \notag\\
   %%%%%%%%%%%%%%%%%
   &&-4 x^3\lambda 
   \left(\rho ^4 (\tau +1)+12 \rho ^3 \tau +2 \rho ^2 \left(3 \lambda ^2-26 \tau +4\right)+\rho  \left(-16 \lambda ^2+32 \tau
   -41\right)\right.\notag\\
   &&+n^2 \left(8 \rho ^3-24 \rho ^2+16 \rho \right)\notag\\
   &&\left.+n \left(8 \rho ^2
   \left(\lambda ^2-6 \tau -1\right)-4 \rho  \left(4 \lambda ^2-8 \tau +7\right)+4 \rho
   ^4 (\tau +1)+8 \rho ^3 \tau +32\right)+32\right) \notag\\
   %%%%%%%%%%%%%%%%%%%%%
   &&-2  x^2\left(2
   \rho ^4 (\tau +1) (6 \tau +7)+\rho ^3 \left(\lambda ^2 (-4 \tau +1)+8 \tau ^2-30 \tau -24\right)\right.\notag\\
   &&+2 \rho ^2
   \left(11 \lambda ^2 (4\tau +1)-8 \tau ^2+6 \tau +3\right)-96 \lambda ^2\notag\\
   &&+2
   \rho  \left(4 \lambda ^4+\lambda ^2(-48 \tau +29 )-4 \tau +2\right)\notag\\
   &&\left.+n^2
   \left(8 \rho ^4 (\tau +1)-8 \rho ^3 \left(\lambda ^2+\tau +2\right)+8 \left(6 \lambda ^2+1\right) \rho
   ^2-48 \lambda ^2 \rho \right)\right.\notag\\
   &&+n \left(2 \rho ^5 (\tau +1)^2+4 \rho ^4 (\tau +1) (4 \tau +3)-2 \rho ^3 \left(4 \lambda^2 (\tau +1)+8 \tau ^2+20 \tau +11\right)\right.\notag\\
   &&\left.\left.+8 \rho ^2 \left( \lambda ^2(12 \tau+7 )+\tau \right)-96 \lambda ^2+8 \rho  \left(\lambda ^4-12 \lambda ^2 \tau
   +3 \lambda ^2+1\right)\right)\right)\notag\\
   %%%%%%%%%%%%%%%%%%%%%%
   &&+2x \lambda  \left(2 \rho ^2 \left(3 \lambda ^2 (4\tau +3 )-16 \tau ^2+6 \tau +2\right)-2 \rho 
   \left( \lambda ^2 (32\tau +1)+8 \tau -4\right)-64 \lambda ^2\right.\notag\\
   &&+n^2 \left(16
   \left(\lambda ^2+1\right) \rho ^2-32 \lambda ^2 \rho +8 \rho ^4 (\tau +1)-8 \rho ^3 (2
   \tau +3)\right)\notag\\
   &&+n \left(\rho ^4
   (\tau +1) (4 \tau +5)+\rho ^3 (-23 \tau -17)+4 \rho ^2 \left(8 \lambda ^2( \tau +1)+4 \tau
   +3\right)\right.\notag\\
   &&\left.\left.-8 \rho  \left( \lambda ^2 (8\tau +3)-2\right)-64 \lambda ^2+4 \rho
   ^4 (\tau +1) (4 \tau +5)-4 \rho ^3 \left(8 \tau ^2+19 \tau +12\right)\right)\right) \notag\\
   %%%%%%%%%%%%%%%%%%%%%%
   &&+\rho ^4 (\tau +1)+12 \rho ^3 \tau +2 \rho ^2 \left(3 \lambda ^2-26 \tau +4\right)+\rho  \left(-16 \lambda ^2+32 \tau
   -41\right)\notag\\
   &&+n^2 \left(8 \rho ^3-24 \rho ^2+16 \rho \right)\notag\\
   &&+n \left(4 \rho
   ^4 (\tau +1)+8 \rho ^3 \tau +8 \rho ^2
   \left(\lambda ^2-6 \tau -1\right)-4 \rho  \left(4 \lambda ^2-8 \tau +7\right)+32\right)+32
\Big),\notag\\
%%%%%%%%%%%%%%%%%%%%%%%%%%%%%%%%%%%%%%%%%%%%%%%%%%%%%%%%%
%%%%%%%%%%%%%%%%%%%%%%%%%%%%%%%%%%%%%%%%%%%%%%%%%%%%%%%%%
%%%%%%%%%%%%%%%%%%%%%%%%%%%%%%%%%%%%%%%%%%%%%%%%%%%%%%%%%
%%%%%%%%%%%%%%%%%%%%%%%%%%%%%%%%%%%%%%%%%%%%%%%%%%%%%%%%%
\widehat{{\cal K}}_{2}(n;x)  & = & -4\Big(
-16 (n+1) (\rho -1)^3 x^7+8 \lambda  (\rho -1)^2 ( 2 n(2 \rho-5)+3 \rho -10) x^6\notag\\
    %%%%%%%%%%%%%%%%%%%%%
&&+8 (\rho
   -1)x^5 \left(\rho ^3 (\tau +1)+\rho ^2 \left(-\lambda ^2+5 \tau -5\right)+\rho  \left(13 \lambda ^2-10 \tau \right)\right.\notag\\
   &&-4
   \left(5 \lambda ^2-\tau -1\right)+n^2 \left(2 \rho ^2-4 \rho +2\right)\notag\\
   &&\left.+n \left(2 \rho ^3 (\tau +1)+\rho ^2
   \left(-2 \lambda ^2+2 \tau -5\right)+\rho  \left(16 \lambda ^2-8 \tau -3\right)-2
   \left(10 \lambda ^2-2 \tau -3\right)\right)\right)
   \notag\\
      %%%%%%%%%%%%%%%%%%%%%
   &&-4 \lambda   x^4\left(\rho ^4 (\tau +1)+\rho ^3 (9 \tau -13)+2 \rho ^2 \left(4 \lambda ^2-37 \tau +8\right)\right.\notag\\
   &&-2 \rho  \left(21 \lambda ^2-52 \tau-18\right)+40 \left(\lambda ^2-\tau -1\right)+n^2 \left(8 \rho ^3-36 \rho ^2+48 \rho
   -20\right)\notag\\
   &&+n \left(4 \rho ^4 (\tau+1)-22 \rho ^3+2 \rho ^2 \left(6 \lambda ^2-30 \tau -1\right)-16 \rho  \left(3
   \lambda ^2-6 \tau -5\right)\right.\notag\\
   &&\left.\left.+20 \left(2 \lambda ^2-2 \tau -3\right)\right)\right)
  \notag\\
   %%%%%%%%%%%%%%%%%%%%%
   &&+2x^3 \left(-4\rho ^4 (\tau +1) (3 \tau +2)+\rho ^3 \left( \lambda ^2(2 \tau -11 )+4 \tau ^2+94 \tau +74\right)\right.\notag\\
   &&+\rho ^2\left( \lambda ^2 (-116\tau +15 )+24 \tau ^2-90 \tau -148\right)\notag\\
   &&\left.-2 \rho 
   \left(10 \lambda ^4- \lambda ^2 (144\tau -60 )+8 \tau ^2-8 \tau -53\right)+8
   \left(5 \lambda ^4-20 \lambda ^2( \tau +1)-3\right)\right.\notag\\
   &&\left.+n^2 \left(8 \rho ^3\left(\lambda ^2+2 \tau +4\right)-8 \rho ^2 \left(9 \lambda ^2+\tau +5\right)+16
   \left(9 \lambda ^2+1\right) \rho -80 \lambda ^2-8 \rho ^4 (\tau +1)\right)\right.\notag\\
   &&+n \left(-2 \rho ^5 (\tau +1)^2-2 \rho ^4 (\tau +1) (7 \tau -3)-2\rho ^3 \left(8 \lambda ^2-16 \tau ^2-42 \tau -33\right)\right.\notag\\
   && -2 \rho ^2 \left( \lambda ^2(60\tau +16)+8 \tau ^2+52 \tau +83\right)\notag\\
   &&\left.\left.-8 \rho  \left(3 \lambda ^4- \lambda
   ^2 (36\tau +33)-4 \tau -15\right)+8 \left(5 \lambda ^4-10 \lambda ^2 (2\tau +3)
-3\right)\right)\right) \notag\\
      %%%%%%%%%%%%%%%%%%%%%
 &&     -2x^2\lambda \left(-\rho ^4 (\tau +1) (4 \tau +1)+\rho ^3 \left(16 \tau^2+88 \tau +59\right)\right.\notag\\
 &&\left.+
      \rho ^2 \left( \lambda ^2 (-28\tau +3 )+40 \tau ^2-170 \tau -235\right)\right.\notag\\
      &&+\rho 
   \left( \lambda ^2 (176\tau +88)-48 \tau ^2+48 \tau +249\right)+8
   \left(\lambda ^4-20 \lambda ^2( \tau +1)-9\right)\notag\\
   &&+n^2 \left(-8 \rho ^4 (\tau +1)+4 \rho ^3 (8 \tau +13)-4 \rho ^2
   \left(6 \lambda ^2+6 \tau +23\right)+48 \left(2 \lambda ^2+1\right) \rho -80 \lambda
   ^2\right)\notag\\
   &&+n \left(-14 \rho ^4 \tau  (\tau +1)+\rho ^3 \left(64 \tau ^2+152 \tau
   +105\right)\right.\notag\\
   &&\left.+\rho ^2 \left(-20 \lambda^2(2 \tau +1)-48 \tau ^2-256 \tau -345\right)+24 \rho  \left(8 \lambda ^2(
   \tau +1)+4 \tau +13\right)\right.\notag\\
   &&\left.\left.+8 \left(\lambda ^4-10 \lambda ^2 (2\tau+3)-9\right)\right) \right)\notag\\
   %%%%%%%%%%%%%%%%%%%%%
   &&+x
   \left(\rho ^3 \left(\lambda ^2 (24 \tau ^2+60 \tau +29)-16 \tau ^3-152 \tau
   ^2-222 \tau -48\right)\right.\notag\\
   &&+\rho ^2 \left(\lambda ^2(16 \tau ^2-228\tau -269)+88 \tau ^2+236 \tau +78\right)\notag\\
   &&-16 \lambda ^2 \left(10 \lambda ^2 (\tau +1)+9\right)+4 \rho  \left( \lambda ^4 (20\tau +12 )+\lambda ^2(24 \tau
   ^2+24\tau +96)-12 \tau -9\right)\notag\\
   &&+n^2 \left(4 \rho
   ^5 (\tau +1)^2-4 \rho ^4 (\tau +1) (\tau +5)+4 \rho^3 \left(8 \lambda ^2 \tau +10 \lambda ^2+4 \tau +7\right)\right.\notag\\
   &&\left.-4 \rho ^2 \left(12 \lambda
   ^2 \tau +32 \lambda ^2+3\right)+48 \lambda ^2 \left(\lambda ^2+2\right) \rho -80 \lambda ^4\right)\notag\\
   &&+n \left(2 \rho ^5 (\tau+1)^2 (4 \tau +1)-4 \rho ^4 (\tau +1) \left(2 \tau ^2+13 \tau +2\right)\right.\notag\\
   &&\left.+2 \rho ^3 \left(\lambda^2 (32\tau ^2+68  \tau +49 )+16 \tau ^2+10 \tau -1\right)\right.\notag\\
   &&-2 \rho ^2\left( \lambda ^2 (48\tau ^2+200\tau +229 )-14 \tau -10\right)
   +4 \rho  \left( \lambda^4 (24\tau +26 )+ \lambda ^2 (48\tau +132)-3\right)\notag\\
&&\left. \left.  -16\lambda ^2 \left(10 \lambda ^2 \tau +15 \lambda ^2+9\right)\right)  +2 \rho^4 (\tau +1) (2 \tau +3) (4 \tau +1) \right) \notag\\
%%%%%%%%%%%%%%%%%%%%%  
&&+\lambda  \left(\rho ^3 \left(16\tau ^3+80 \tau ^2+122 \tau +51\right)+\rho ^2 \left(\lambda ^2 (16\tau ^2+68  \tau +71 )-88 \tau ^2-188 \tau
   -87\right)\right.\notag\\
   &&+2 \rho  \left( \lambda ^2 (16\tau ^2-16  \tau -49)+24\tau +18\right)+16 \lambda ^2 \left( 2\lambda ^2 (\tau +1)+3\right)\notag\\
   %%%%
&&+n^2 \left(4 \rho ^4 (\tau +1) (\tau +2)-4 \rho ^3 (4 \tau +5)+4 \rho ^2 \left(4 \lambda ^2 \tau +6 \lambda ^2+3\right)
-32 \lambda ^2 \rho+16 \lambda ^4 \right)\notag\\
%%%
 &&+n \left(\rho ^4 (\tau +1) \left(8 \tau ^2+28
   \tau +19\right)+\rho ^3 \left(-32 \tau ^2-36 \tau -23\right)\right.\notag\\
   &&+4\rho ^2 \left( \lambda ^2(8 \tau ^2+24 \tau +25)-7 \tau-2\right)\notag\\
 &&\left.\left.-4 \rho  \left(\lambda ^2(16 \tau +36 )-3\right)+16 \lambda ^2
   \left(2 \lambda ^2 \tau +3 \lambda ^2+3\right)\right)\right)
   \Big),\notag\\
%%%%%%%%%%%%%%%%%%%%%%%%%%%%%%%%%%%%%%%%%%%%%%%%%%%%%%%%%
%%%%%%%%%%%%%%%%%%%%%%%%%%%%%%%%%%%%%%%%%%%%%%%%%%%%%%%%%
%%%%%%%%%%%%%%%%%%%%%%%%%%%%%%%%%%%%%%%%%%%%%%%%%%%%%%%%%
%%%%%%%%%%%%%%%%%%%%%%%%%%%%%%%%%%%%%%%%%%%%%%%%%%%%%%%%%
\widehat{{\cal K}}_{1}(n;x) & = &-4\Big(
-16 (n+1) (\rho -1)^3 x^6+16 \lambda  (\rho -1)^2 ( n(3 \rho -7)+2 \rho -7) x^5\notag\\
%%%%%%%%%%%%%%%%%%%%%%%%%%%%%%%
&&+8 (\rho -1) x^4\left(\rho ^3 (\tau +1)+\rho ^2 \left(-2 \lambda ^2-3 \tau -14\right)+\rho  \left(24 \lambda ^2+10 \tau
   +21\right)\right.\notag\\
   &&-4 \left(9 \lambda ^2+2 \tau +2\right)+n^3 \left(\rho ^2-\rho \right)+n^2\left(-\rho ^2+5 \rho -4\right)\notag\\
   &&\left.+n \left(4 \rho ^3 (\tau +1)-2 \rho ^2 \left(2 \lambda ^2+6 \tau
   +11\right)+2 \rho  \left(15 \lambda ^2+8 \tau +15\right)-4 \left(9 \lambda ^2+2 \tau
   +3\right)\right)\right) \notag\\
   %%%%%%%%%%%%%%%%%%%%%%%%%%%%%%%
   &&-4 \lambda x^3 \left(2 \rho ^4 (\tau +1)-5 \rho ^3 (4 \tau +11)+2 \rho ^2 \left(10 \lambda ^2+33 \tau +96\right)\right.\notag\\
   &&+\rho  \left(-96 \lambda ^2-112 \tau
   -203\right)+8 \left(11 \lambda ^2+8 \tau +8\right)\notag\\
   &&+n^3 \left(4 \rho ^3-12 \rho ^2+8 \rho \right)+n^2 \left(-3 \rho ^3+23 \rho ^2-52 \rho +32\right)\notag\\
   &&+n \left(10 \rho ^4 (\tau +1)-2 \rho ^3 (31\tau +53)+\rho ^2\left(28 \lambda ^2+140 \tau +283\right)\right.\notag\\
   &&\left.\left.+\rho  \left(-108 \lambda ^2-152 \tau -283\right)+8 \left(11 \lambda ^2+8 \tau +12\right)\right)\right) \notag\\
   %%%%%%%%%%%%%%%%%%%%%%%%%%%%%%
   &&+4  x^2\left(\rho ^4 (\tau+1) (2 \tau +5)+\rho ^3 \left(- \lambda ^2(12 \tau +25)-10 \tau ^2-44 \tau -53\right)\right.\notag\\
   &&+\rho ^2\left(\lambda ^2 (54 \tau +156 )+8 \tau ^2+81 \tau +115\right)+\rho  \left(-28
   \lambda ^4- \lambda ^2 120 \tau +239 )-44 \tau -91\right)\notag\\
   &&+4 \left(13 \lambda^4+24 \lambda ^2 (\tau +1)+6\right)\notag\\
   &&+n^3 \left(-2 \rho ^4 (\tau +1)+2 \rho ^3 \left(\lambda^2+\tau +2\right)-2 \left(6 \lambda ^2+1\right) \rho ^2+12 \lambda ^2 \rho \right)\notag\\
   &&+n^2 \left(\rho ^4 (\tau +1)+\rho ^3 \left(-\lambda ^2-11 \tau -14\right)+5 \rho ^2
   \left(2 \lambda ^2+2 \tau +5\right)\right.\notag\\
   &&\left.-12 \left(4 \lambda ^2+1\right) \rho +48 \lambda^2\right)\notag\\
   &&+n \left(-3 \rho ^5 (\tau +1)^2+\rho ^4 (\tau +1) (11 \tau +41)\right.\notag\\
   &&+\rho ^3 \left(- \lambda ^2 (30\tau +45 )-16 \tau ^2-132 \tau -143\right)\notag\\
  && +\rho ^2 \left( \lambda ^2 (112\tau +225)+8 \tau ^2+130 \tau +199\right)\notag\\
   &&\left.\left.+\rho  \left(-32 \lambda ^4- \lambda ^2 (168\tau +323
 )-44 \tau -118\right)+4 \left(13 \lambda ^4+\lambda ^2(24  \tau +36 )+6\right)\right)\right)\notag\\
   %%%%%%%%%%%%%%%%%%%%%%%%%%%%%%%%%%%%%
   &&-2 \lambda x  \left(2 \rho ^4 (\tau +1) (2 \tau +5)-3 \rho ^3 \left(8 \tau
   ^2+36 \tau +43\right)\right.\notag\\
   &&+\rho ^2 \left(\lambda ^2 (28 \tau +68 )+32 \tau ^2+252 \tau +339\right)-2 \rho 
   \left( \lambda ^2 (48\tau +101)+88 \tau +158\right)\notag\\
   &&+8 \left(3 \lambda ^4+16
   \lambda ^2 (\tau +1)+12\right)\notag\\
   &&+n^3 \left(-4 \rho ^4 (\tau +1)+4 \rho ^3 (2 \tau +3)-8 \left(\lambda ^2+1\right) \rho
   ^2+16 \lambda ^2 \rho \right)\notag\\
   &&+n^2 \left(\rho ^4 (\tau +1)-5 \rho ^3 (4 \tau +5)-2
   \rho ^2 \left(\lambda ^2-20 \tau -36\right)-24 \left(\lambda ^2+2\right) \rho +64
   \lambda ^2\right)\notag\\
   &&+n \left(2 \rho ^4 (\tau +1) (7 \tau +22)+\rho ^3 \left(-40 \tau
   ^2-271 \tau -272\right)\right.\notag\\
   &&+4 \rho ^2 \left(\lambda ^2(14 \tau +23 )+8 \tau ^2+101 \tau +133\right)-2 \rho  \left( \lambda
   ^2 (72\tau +133)+88 \tau +200\right)\notag\\
   &&\left.\left.+8 \left(3 \lambda ^4+\lambda ^2 (16 \tau
   +24)+12\right)\right)\right) \notag\\
   %%%%%%%%%%%%%%%%%%%%%%%%%%%%%%%%%%%%
   &&-8 \rho ^5 (\tau +1)^2+4 \rho ^4 (\tau +1) (6 \tau
   +19)-2 \lambda ^2 \rho ^3 \left(4 \tau ^2+20 \tau +23\right)\notag\\
   &&+2 \lambda ^2 \rho ^2 \left(16
   \tau ^2+90 \tau +109\right)-4 \lambda ^2 \rho  \left( \lambda ^2(4 \tau +7 )+44 \tau +67\right)\notag\\
   &&+32 \lambda ^2 \left(2 \lambda ^2 (\tau +1)+3\right)\notag\\
   &&+n^3
   \left(2 \rho ^5 (\tau +1)^2-4 \rho ^4 (\tau +1) +2 \rho ^3 \left(4 \lambda ^2 \tau +4 \lambda ^2+1\right)-8
   \lambda ^2 \rho ^2+8 \lambda ^4 \rho\right)\notag\\
   &&+n^2 \left(6 \rho ^4 (\tau +1) (2 \tau +5)+\rho ^3 \left(4 \lambda ^2 \tau +11 \lambda ^2-42 \tau -60\right)\right.\notag\\
   &&\left.+2 \rho ^2
   \left( \lambda ^2(20 \tau +22 )+15\right)+8 \left(\lambda ^2-6\right) \lambda
   ^2 \rho+32\lambda ^4 \right)\notag\\
   &&+n \left(\rho ^3 \left( -\lambda ^2 (16\tau ^2+86  \tau +63 )-84 \tau -128\right)+4 \rho ^2 \left(
   \lambda ^2 (8\tau ^2+72  \tau +79 )+15\right)\right.\notag\\
   &&\left.-4 \lambda ^2 \rho 
   \left( \lambda ^2 (8\tau +9)+44 \tau +82\right)+32 \lambda ^2 \left( \lambda
   ^2 (2\tau +3 )+3\right)\right)
\Big) ,\notag\\
%%%%%%%%%%%%%%%%%%%%%%%%%%%%%%%%%%%%%%%%%%%%%%%%%%%%%%%%%
%%%%%%%%%%%%%%%%%%%%%%%%%%%%%%%%%%%%%%%%%%%%%%%%%%%%%%%%%
%%%%%%%%%%%%%%%%%%%%%%%%%%%%%%%%%%%%%%%%%%%%%%%%%%%%%%%%%
%%%%%%%%%%%%%%%%%%%%%%%%%%%%%%%%%%%%%%%%%%%%%%%%%%%%%%%%%
\widehat{{\cal K}}_{0}(n;x)  & = &-8(n+1)\Big(
8 (n+1)^2 (\rho -1)^3 x^5\notag\\
&&-4x^4 \lambda  (\rho -1)^2 \left(n^2 (4 \rho -10)+n (5 \rho -20)+2 \rho -12\right) \notag\\
%%%%%%%%%%%%%%%%%%%
&&+4 x^3(\rho -1) \left(\rho ^3 (-\tau -1)+\rho ^2 (3 \tau +14)+\rho  \left(-10 \lambda ^2-10 \tau -21\right)\right.\notag\\
&&+8 \left(3 \lambda ^2+\tau +1\right)+n^2 \left(-2 \rho ^3 (\tau +1)+2 \rho ^2 \left(\lambda ^2+\tau +4\right)-2 \left(8 \lambda ^2+3\right) \rho +20
   \lambda ^2\right)\notag\\
   &&\left.+n \left(-\rho ^3 (\tau +1)+\rho ^2 \left(\lambda ^2-5 \tau +14\right)+\rho  \left(-23 \lambda ^2+6 \tau -17\right)+4 \left(10 \lambda^2+1\right)\right)\right) \notag\\
   %%%%%%%%%%%%%%%%%%%
   &&+2 \lambda x^2 \left(\rho ^3
   (-12 \tau -25)+2 \rho ^2 (20 \tau +61)+\rho  \left(-28 \lambda ^2-76 \tau -145\right)\right.\notag\\
   &&+8 \left(5 \lambda ^2+6 \tau +6\right)\notag\\
   &&+n^2
   \left(4 \rho ^4 (\tau +1)-16 \rho ^3 (\tau +2)+4 \rho ^2 \left(3 \lambda ^2+3 \tau +16\right)-12 \left(4 \lambda ^2+3\right)
   \rho +40 \lambda ^2\right)\notag\\
   &&+n \left(-\rho ^4 (\tau +1)+\rho ^3 (7 \tau -27)+6 \rho ^2
   \left(2 \lambda ^2-7 \tau +19\right)-2 \rho  \left(39 \lambda ^2-18 \tau +55\right)\right.\notag\\
   &&\left.\left.+8\left(10 \lambda ^2+3\right)\right)\right) \notag\\
      %%%%%%%%%%%%%%%%%%%
   &&+x\left(-2 \rho ^4 (\tau +1) (2 \tau +5)+2 \rho ^3 \left(10 \tau
   ^2+44 \tau +53\right)\right.\notag\\
   &&-2 \rho ^2 \left(\lambda ^2 (14 \tau +34)+8 \tau ^2+81 \tau +115\right)+2 \rho 
   \left( \lambda ^2(44 \tau +94)+44 \tau +91\right)\notag\\
   &&-24 \left(\lambda ^4+4
   \lambda ^2 (\tau +1)+2\right)\notag\\
   &&+n^2 \left(2 \rho ^5 (\tau+1)^2-2 \rho ^4 (\tau +1) (\tau +7)+2 \rho ^3 \left(8
   \lambda ^2 \tau +12 \lambda ^2+6 \tau +11\right)\right.\notag\\
   &&\left.-2 \rho ^2 \left(12 \lambda ^2 \tau
   +44 \lambda ^2+5\right)+24 \lambda ^2 \left(\lambda ^2+3\right) \rho -40 \lambda ^4\right)\notag\\
   %%%%
    &&+n \left(-2 \rho ^5 (\tau +1)^2+2 \rho ^4
   (\tau +1) (7 \tau +12)+\rho ^3 \left(2 \lambda ^2 \tau +21
   \lambda ^2-12 \tau ^2-92 \tau -48\right)\right.\notag\\
   &&+\rho ^2 \left(36 \lambda ^2 \tau -133 \lambda
   ^2+58 \tau +32\right)-16 \lambda ^2 \left(5 \lambda ^2+3\right)\notag\\
   &&\left.\left.+2 \rho  \left(18
   \lambda ^4-36 \lambda ^2 \tau +94 \lambda ^2-3\right)\right)\right) \notag\\
   %%%%%%%%%%%%%%%%%%%%%%%%%%%%
   &&+\lambda  \left(\rho ^3 \left(-4 \tau ^2-20 \tau -23\right)+\rho ^2
   \left(16 \tau ^2+90 \tau +109\right)-2 \rho  \left( \lambda ^2 (4\tau +7 )+44 \tau +67\right)\right.\notag\\
   &&\left.+16 \left(2 \lambda ^2(
   \tau +1)+3\right)\right.\notag\\
   &&+n^2 \left(2 \rho ^4 (\tau +1) (\tau +3)-4 \rho ^3 (3
   \tau +4)+2 \rho ^2 \left(4 \lambda ^2 \tau
   +8 \lambda ^2+5\right)-24 \lambda ^2 \rho+8 \lambda ^4 \right)\notag\\
   &&+n \left( -\rho ^4 (\tau +1) (2 \tau +3)+\rho ^3 \left(12
   \tau ^2+46 \tau +15\right)+\rho ^2 \left(\lambda ^2(4 \tau +29)-58 \tau-18\right)\right.\notag\\
   &&\left.\left. 
   +2 \rho  \left(12 \lambda ^2 \tau -26 \lambda ^2+3\right)
   +16 \lambda ^2\left(\lambda ^2+1\right) \right)\right)\Big) .\notag
\end{eqnarray}
%%%%%%%%%%%%%%%%%%%%%%%%%%%%%%%%%%%%%%%%%%%%%%%%%%%%%%%%%

%%%%%%%%%%%%%%%%%%%%%%%%%%%%%%%%%%%%%%%%%%%%%%%%%%%%%%%%%%
\subsubsection{Classical Hermite family}
%%%%%%%%%%%%%%%%%%%%%%%%%%%%%%%%%%%%%%%%%%%%%%%%%%%%%%%%%%

The classical Hermite monic orthogonal sequence is recovered as the particular case of the preceding family corresponding to 
 $\tau=0$, $\lambda =0$, and $\rho=1$. The input data for the algorithm {\it HoDELH} are as follows:
 
\noindent  - the recurrence coefficients, 
\begin{eqnarray}
&&\beta_{n}=0, \quad  \gamma_{n+1}={{n+1}\over{2}},\quad n\ge 0; \label{RC_CLHermite}
\end{eqnarray}
- the coefficients of the functional equation, 
\begin{eqnarray}
&&\Phi(x)=1,\quad\psi(x)=2x ,\quad B(x)=0;\label{FE_CLHermite}
\end{eqnarray}
-  the coefficients of the main Laguerre-Hahn structure relation, 
\begin{eqnarray}
&&C_{n}(x)=-2x,\quad D_{n}(x)=-2,\quad n\ge 0.\label{SE_CLHermite}
\end{eqnarray}

%%%%%%%%%%%%%%%%%%%%%%%%%%%%%%%%%%%%%%%%%%%%%%%%%%%%%%
\subsubsection{Structure relations}
%%%%%%%%%%%%%%%%%%%%%%%%%%%%%%%%%%%%%%%%%%%%%%%%%%%%%%

%%%%%%%%%%%%%%%%%%%%%%%%%%%%%%%%%%%%%%%%%%%%%%%%%%%%%%
\begin{proposition}\label{Prop_SR_CLHermite}
If \(\{P_n\}_{n \geq 0}\) is the classical Hermite monic orthogonal polynomial sequence with recurrence coefficients given by \eqref{RC_CLHermite}, then, for $n\geq 0$ and $i\geq 1$, the coefficients of the structure relation \eqref{ST_RR_i}-\eqref{RR_F} in Theorem \ref{th:0.1} satisfy 
\begin{eqnarray}
&& H_{1}(x;n)=n+1,  \label{H1-Hermite}\\
&& H_{i+1}(x;n) = 2x H_i (x;n)+  H_i' (x;n).\label{RR_CLHermite}\\
&& M_{0,1}(x;n)=0, \label{H1-M01}\\
&& M_{0,i+1}(x;n) =  M_{0,i}'(x;n)+2H_i(x;n),\label{H1-RR_M0}\\
&& M_{k,i+1}(x;n) = M_{k-1,i}(x;n) +  M_{k,i}'(x;n),\  1 \le k \le i. \label{RR_Mk_CLHermite}\\
&& M_{i,i}(x;n) =1,\ i\geq 1.  \label{RR_Mii_CLHermite}
\end{eqnarray}
Moreover, for $n\geq 0$,
\begin{eqnarray}
&& M_{i-1,i}(x;n) =0,\ i\geq 1. \label{RR_Mi1i_CLHermite}\\
&& M_{i-2,i}(x;n) =2(n+1),\ i\geq 2. \label{RR_Mi2i_CLHermite}\\
&& M_{i-3,i}(x;n) =4(n+1)x,\ i\geq 3. \label{RR_Mi3i_CLHermite}
\end{eqnarray}
Consequently, the structure relation of index $i\geq 1$ takes the form
\begin{eqnarray}
H_i(x;n) P_n(x) & = & P_{n+1}^{[i]}(x) + 2(n+1) P_{n+1}^{[i-2]}(x)+ 4(n+1)x P_{n+1}^{[i-3]}(x)\notag\\
&&+\sum_{k=0}^{i-4} M_{k,i}(x;n) P_{n+1}^{[k]}(x), \ n\geq 0. \label{SRi_CLHermite}
\end{eqnarray}
\end{proposition}
%%%%%%%%%%%%%%%%%%%%%%%%%%%%%%%%%%%%%%%%%%%%%%%%%%%%%%
\begin{proof}
The formulas \eqref{H1-Hermite}--\eqref{RR_Mk_CLHermite} follow directly from the corresponding identities in Corollary \ref{Cor_STR_SC} by using the characteristic elements of the sequence given in \eqref{RC_CLHermite}, \eqref{FE_CLHermite}, and \eqref{SE_CLHermite}. The formula \eqref{RR_Mii_CLHermite} follows from \eqref{RR_Mii}  in Theorem \ref{th:0.1}, and \eqref{FE_CLHermite}.

For the additional coefficients, the proofs are analogous in all cases and will be established by induction.
Let us begin with \eqref{RR_Mi1i_CLHermite}.  
The formula \eqref{H1-M01} is precisely  \eqref{RR_Mi1i_CLHermite} for $i=1$.
Now assume that \eqref{RR_Mi1i_CLHermite} holds for some fixed $i \geq 1$. Setting $k=i$ in \eqref{RR_Mk_CLHermite}, we obtain
$M_{i,i+1}(x;n)=M_{i-1,i}(x;n)$,
since $M_{i,i}'(x;n)=0$ by \eqref{RR_Mii_CLHermite}. By the induction hypothesis, $M_{i-1,i}(x;n)=0$, therefore
$M_{i,i+1}(x;n)=0$, as required. 

We now prove \eqref{RR_Mi2i_CLHermite}.
Setting $i=1$ in \eqref{H1-RR_M0}, we obtain
$M_{0,2}(x;n)=M_{0,1}'(x;n)+2H_1(x;n)=2(n+1)$,
which is precisely \eqref{RR_Mi2i_CLHermite} for $i=2$. Now assume that
$M_{i-2,i}(x;n)=2(n+1)$
for some fixed $i\geq 2$. Taking $k=i-1$ in \eqref{RR_Mk_CLHermite}, we obtain
$M_{i-1,i+1}(x;n)=M_{i-2,i}(x;n)$,
since $M_{i-1,i}'(x;n)=0$ by \eqref{RR_Mi1i_CLHermite}. Hence,
$M_{i-1,i+1}(x;n)=2(n+1)$,
which completes the induction.

Finally, let us prove \eqref{RR_Mi3i_CLHermite}.
Setting $i=2$ in \eqref{H1-RR_M0}, we obtain \eqref{RR_Mi2i_CLHermite} for $i=3$. In fact, 
$M_{0,3}(x;n)=M_{0,2}'(x;n)+2H_2(x;n)=4(n+1)x$, since $M_{0,2}'(x;n)=0$, and $H_2(x;n)=2(n+1)x$  by \eqref{RR_CLHermite} for $i=2$. 
Now assume that $M_{i-3,i}(x;n)=4(n+1)x$ for some fixed $i\geq 3$. Taking $k=i-2$ in \eqref{RR_Mk_CLHermite}, we obtain $M_{i-2,i+1}(x;n)=M_{i-3,i}(x;n)$, since $M_{i-2,i}'(x;n)=0$ by \eqref{RR_Mi2i_CLHermite} which completes the induction.

The structure relation \eqref{SRi_CLHermite} follows directly from the coefficients just proved.
\end{proof}
%%%%%%%%%%%%%%%%%%%%%%%%%%%%%%%%%%%%%%%%%%%%%%%%%%%%%%

The following table lists the coefficients of the structure relations of indices $i=1,2,3$, and $4$ satisfied by the classical Hermite monic orthogonal polynomials. These coefficients are obtained by symbolic computation and were first reported in \cite{Article-1-NA}.
\begin{center}
\begin{tabular}{ l || l || l | l | l | l | l | l | l | l | | }
$i$ & $H_i(x;n)$ & $M_{0,i}(x;n)$ & $M_{1,i}(x;n)$ & $M_{2,i}(n)$  &$M_{3,i}$ & $M_{4,i}$ \\ \hline \hline
1 & $n+1$ & 0 & 1 & & &   \\ \hline
2 & $2(n+1)x$ & $2(n+1)$ & 0 & 1 & &   \\ \hline
3 & $2(n+1)(x^2+1)$ & $4(n+1)x$ & $2(n+1)$ & 0 & 1 &  \\ \hline
4 &  $4(n+1)x(2x^2+3)$ &  $8(n+1)(x^2+3)$ &  $4(n+1)x$ & $2(n+1)$ & 0 & 1   \\ \hline
\end{tabular}
\end{center}
We observe that the coefficients $M_{k,i}(x;n)$ are identical along each of the first four descending diagonals, counted from right to left. This property does not extend to the remaining diagonals.

Next, we present the coefficients of the structure relations and differential equations for $i=5,\ldots,10$, obtained by applying the algorithm \textit{HoDELH} with symbolic computations.
 
 \vspace{0.5cm}
 
\noindent $\bullet$ \noindent {\bf Structure relation for} $i=5$
 \begin{eqnarray}
&& H_5(x;n)=4 (n+1) \left(4 x^4+12 x^2+3\right),\notag\\
&& M_{0,5}(x;n)= 8 (n+1) x \left(2 x^2+5\right),\ M_{1,5}(x;n)=4 (n+1) \left(2 x^2+3\right).\notag
 \end{eqnarray}
 
\noindent $\bullet$ \noindent {\bf Structure relation for} $i= 6$
 \begin{eqnarray}
&& H_6(x;n)=8 (n+1) x \left(4 x^4+20 x^2+15\right),\notag\\
&& M_{0,6}(x;n)= 16 (n+1) \left(x^2+4\right) \left(2 x^2+1\right),\ M_{1,6}(x;n)=8 (n+1) x \left(2 x^2+7\right),\notag\\
&&M_{2,6}(x;n)=8 (n+1) \left(x^2+2\right).\notag
 \end{eqnarray}
 
\noindent $\bullet$ \noindent {\bf Structure relation for} $i= 7$
 \begin{eqnarray}
&& H_7(x;n)=8 (n+1) \left(8 x^6+60 x^4+90 x^2+15\right),\notag\\
&& M_{0,7}(x;n)= 16 (n+1) x \left(2 x^2+3\right) \left(2 x^2+11\right),\  M_{1,7}(x;n)=8 (n+1) \left(4 x^4+24 x^2+15\right),\notag\\
&&M_{2,7}(x;n)=8 (n+1) x \left(2 x^2+9\right),\ M_{3,7}(x;n)=4 \left(2 x^2+5\right).\notag
 \end{eqnarray}
 
\noindent $\bullet$ \noindent {\bf Structure relation for} $i= 8$
 \begin{eqnarray}
&& H_8(x;n)=16 (n+1) x \left(8 x^6+84 x^4+210 x^2+105\right),\notag\\
&& M_{0,8}(x;n)= 32 (n+1) \left(4 x^6+40 x^4+87 x^2+24\right),\notag\\
&&  M_{1,8}(x;n)=16 (n+1) x \left(4 x^4+36 x^2+57\right),\ M_{2,8}(x;n)=16 (n+1) \left(2 x^4+15 x^2+12\right),\notag\\
&& M_{3,8}(x;n)=8 x \left(2 x^2+11\right),\ M_{4,8}(x;n)=8 (n+1) \left(x^2+3\right).\notag
 \end{eqnarray}

\noindent $\bullet$ \noindent {\bf Structure relation for} $i= 9$
 \begin{eqnarray}
&& H_9(x;n)=16 (n+1) \left(16 x^8+224 x^6+840 x^4+840 x^2+105\right),\notag\\
&& M_{0,9}(x;n)= 32 (n+1) x \left(8 x^6+108 x^4+370 x^2+279\right),\notag\\
&&  M_{1,9}(x;n)=16 (n+1) \left(2 x^2+7\right) \left(4 x^4+36 x^2+15\right),\notag\\
&& M_{2,9}(x;n)=16 (n+1) x \left(4 x^4+44 x^2+87\right),\ M_{3,9}(x;n)=8 \left(4 x^4+36 x^2+35\right),\notag\\
&& M_{4,9}(x;n)=8 (n+1) x \left(2 x^2+13\right),\ M_{5,9}(x;n)=4 (n+1) \left(2 x^2+7\right),\notag
 \end{eqnarray}

\noindent $\bullet$ \noindent {\bf Structure relation for} $i=10$
 \begin{eqnarray}
&& H_{10}(x;n)=332 (n+1) x \left(16 x^8+288 x^6+1512 x^4+2520 x^2+945\right),\notag\\
&& M_{0,10}(x;n)=64 (n+1) \left(8 x^8+140 x^6+690 x^4+975 x^2+192\right) ,\notag\\
&&  M_{1,10}(x;n)=32 (n+1) x \left(8 x^6+132 x^4+570 x^2+561\right),\notag\\
&& M_{2,10}(x;n)=32 (n+1) \left(4 x^6+60 x^4+207 x^2+96\right),\notag\\
&& M_{3,10}(x;n)=16 (n+1) x \left(4 x^4+52 x^2+123\right),\notag\\
&& M_{4,10}(x;n)=16 (n+1) \left(2 x^4+21 x^2+24\right),\notag\\
&& M_{5,10}(x;n)=8 (n+1) x \left(2 x^2+15\right),\notag\\
&& M_{6,10}(x;n)=8 (n+1) \left(x^2+4\right).\notag
 \end{eqnarray}
 
%%%%%%%%%%%%%%%%%%%%%%%%%%%%%%%%%%%%%%%%%%%%%%%%%%%%%%
\subsubsection{Differential equations}
%%%%%%%%%%%%%%%%%%%%%%%%%%%%%%%%%%%%%%%%%%%%%%%%%%%%%%

%%%%%%%%%%%%%%%%%%%%%%%%%%%%%%%%%%%%%%%%%%%%%%
\begin{proposition}\label{Pro_DEi_CLHermite}
If \(\{P_n\}_{n \geq 0}\) is the classical Hermite monic orthogonal polynomial sequence with recurrence coefficients given by \eqref{RC_CLHermite}, then \(\{P_n\}_{n \geq 0}\) satisfies the following homogeneous linear differential equation of order 
$i\geq 2$:
\begin{eqnarray}
&& P_{n+1}^{[i]}(x) + 2(n+1) P_{n+1}^{[i-2]}(x)+ 4(n+1)x P_{n+1}^{[i-3]}(x)+\sum_{k=2}^{i-4} M_{k,i}(x;n) P_{n+1}^{[k]}(x)\notag\\
&&+\left(M_{1,i}(x;n)-\frac{1}{n+1} H_i(x;n) \right)P_{n+1}^{[1]}(x)+M_{0,i}(x;n) P_{n+1}(x)=0, \ n\geq 0, \notag
\end{eqnarray}
where the coefficients $H_i(x;n) $ and $M_{0,i}(x;n)$ are given in Proposition \ref{Prop_SR_CLHermite}.
\end{proposition}
%%%%%%%%%%%%%%%%%%%%%%%%%%%%%%%%%%%%%%%%%%%%%
\begin{proof}
The first structure relation is \cite{Article-1-NA}
\begin{equation}
(n+1)P_{n}(x) =P'_{n+1}(x)\Leftrightarrow P_n(x)=\frac{1}{n+1}P'_{n+1}(x),\label{Appell_CLHermite}
\end{equation}
which expresses the Appell property of the Hermite sequence.
Substituting the second equality in \eqref{Appell_CLHermite} into the structure relation of index $i$, \eqref{SRi_CLHermite}, eliminates $P_n(x)$ and yields the required differential equation. The coefficient of $P'_{n+1}(x)$ is therefore
$${\cal K}_1(x;n)=M_{1,i}(x;n)-\frac{1}{n+1} H_i(x;n).$$
\end{proof}
%%%%%%%%%%%%%%%%%%%%%%%%%%%%%%%%%%%%%%%%%%%%%%
%%%%%%%%%%%%%%%%%%%%%%%%%%%%%%%%%%%%%%%%%%%%%%
\begin{remark}
This proposition provides the first and most straightforward approach for eliminating 
$P_n(x)$ from the structure relation of index $i$ thereby yielding an 
 $i$th-order differential equation. By applying the procedure described in Proposition \ref{Cor_DEij_SC}, it is also possible to derive additional $i-2$ differential equations of order $i$.
\end{remark}
%%%%%%%%%%%%%%%%%%%%%%%%%%%%%%%%%%%%%%%%%%%%%%

Next, we present the differential equations corresponding to the orders $i=2,3,4,5,$ and $9$, and to $1\leq j<i-1$, as computed symbolically by the algorithm \textit{HoDELH}. 

\vspace{0.25cm}
\noindent $\bullet$ \noindent {\bf Differential equation of order} $i=2$
\begin{eqnarray}
&&j=1,\quad   P''_{n+1}(x)-2xP'_{n+1}(x)+2(n+1)P_{n+1}(x)=0.\notag
\end{eqnarray}

We observe that this equation is the well-known second-order homogeneous linear differential equation characterizing the Hermite polynomial sequence \cite{Maroni-1994}.

\vspace{0.25cm}

\noindent $\bullet$ \noindent {\bf Differential equations of order} $i=3$
\begin{eqnarray}
&& j=1,\quad P^{(3)}_{n+1}(x)+2 (-2x^2+n)P'_{n+1}(x)+4 (n+1) xP_{n+1}(x)=0,\notag\\
&& j=2,\quad xP^{(3)}_{n+1}(x)-(2 x^2+1)P''_{n+1}(x)+2 (n+1) xP'_{n+1}(x)-2 (n+1)P_{n+1}(x)=0.\notag
\end{eqnarray}
%%%%%%%%%%%%%%%%%%%%%%%%%%%%%%%%%%%%

\noindent $\bullet$ \noindent {\bf Differential equations of order} $i=4$
\begin{eqnarray}
 j=1, \quad && P^{(4)}_{n+1}(x)+2 ( n+1) P''_{n+1}(x)+4 x(-2x^2+n-2) P'_{n+1}(x)\notag\\
&&+8 (n+1) (x^2+1)P_{n+1}(x)=0,\notag\\
 j=2, \quad &&  P^{(4)}_{n+1}(x)+2 \left(-2 x^2+n-2\right)P''_{n+1}(x)+4 (n+1) x P'_{n+1}(x)\notag\\
&&-4 (n+1)P_{n+1}(x)=0.\notag\\
j=3, \quad && ( 2 x^2+1)P^{(4)}_{n+1}(x)-2 x \left(2 x^2+3\right)P^{(3)}_{n+1}(x)+2 (n+1) \left(2 x^2+1\right)P''_{n+1}(x)\notag\\
&&-8 (n+1) x P'_{n+1}(x)+8 (n+1)P_{n+1}(x)=0.\notag
\end{eqnarray}
%%%%%%%%%%%%%%%%%%%%%%%%%%%%%%%%%%%%

\noindent $\bullet$ \noindent {\bf Differential equations of order} $i=5$
\begin{eqnarray}
 j=1, \ && P^{(5)}_{n+1}(x)+2 (n+1)P^{(3)}_{n+1}(x)+4 (n+1) xP''_{n+1}(x)\notag\\
&& +4\left(-4 x^4+2 (n-5) x^2+3 n\right)P'_{n+1}(x)+8 (n+1) x \left(2 x^2+5\right)P_{n+1}(x)=0,\notag\\ 
%%%%%%%%%%%%%%%%%%
 j=2, \ && xP^{(5)}_{n+1}(x)+2 (n+1) xP^{(3)}_{n+1}(x)+2 \left(-4 x^4+2 (n-5) x^2-3\right)P''_{n+1}(x)\notag\\
&& +4 (n+1) x \left(2 x^2+3\right)P'_{n+1}(x)-4 (n+1) \left(2 x^2+3\right)P_{n+1}(x)=0,\notag\\ 
%%%%%%%%%%%%%%%%%%
j=3, \ &&(2 x^2+1) P^{(5)}_{n+1}(x)+2\left(-4 x^4+2 (n-5) x^2+n-2 \right)P^{(3)}_{n+1}(x)\notag\\
&&+4 (n+1) x \left(2 x^2+1\right)P''_{n+1}(x) -16 (n+1) x^2P'_{n+1}(x)\notag\\
&&+16 (n+1) xP_{n+1}(x)=0,\notag\\
%%%%%%%%%%%%%%%%%%
j=4, \ &&x \left(2 x^2+3\right)P^{(5)}_{n+1}(x)-\left(4 x^4+12 x^2+3\right)P^{(4)}_{n+1}(x)\notag\\
&&+2 (n+1) x \left(2 x^2+3\right)P^{(3)}_{n+1}(x)\notag\\
&&-6 (n+1) \left(2 x^2+1\right)P''_{n+1}(x) +24 (n+1) xP'_{n+1}(x)\notag\\
&&-24 (n+1)P_{n+1}(x)=0.\notag
\end{eqnarray}
%%%%%%%%%%%%%%%%%%%%%%%%%%%%%%%%%%%%
 
 \noindent $\bullet$ \noindent {\bf Differential equations of order} $i=9$
 \begin{eqnarray}
 j=1, \  && P^{(9)}_{n+1}(x) +2 (n+1)P^{(7)}_{n+1}(x)+ 4 (n+1) xP^{(6)}_{n+1}(x)\notag\\
 &&+4 (n+1) \left(2 x^2+7\right)P^{(5)}_{n+1}(x) + 8 (n+1) x \left(2 x^2+13\right)P^{(4)}_{n+1}(x)\notag\\
&&+8 (n+1) \left(4 x^4+36 x^2+35\right)P^{(3)}_{n+1}(x)\notag\\
&& +16 (n+1) x \left(4 x^4+44 x^2+87\right)P''_{n+1}(x) \notag\\
&&+16\left(-16 x^8+8 (n-27) x^6+20 (5 n-37) x^4+6 (47 n-93) x^2+105 n\right)P'_{n+1}(x)\notag\\
&&+32 (n+1) x \left(8 x^6+108 x^4+370 x^2+279\right) P_{n+1}(x)=0,\notag\\
%%%%%%%%%%%%%%%%%%%%%%%%%%%%%%%%%%%%%%%%%%%%%%%%
j=2, \  && xP^{(9)}_{n+1}(x) +2 (n+1)P^{(7)}_{n+1}(x)+4 (n+1) x^2P^{(6)}_{n+1}(x)\notag\\
 &&+4 (n+1) x \left(2 x^2+7\right)P^{(5)}_{n+1}(x) +8 (n+1) x^2 \left(2 x^2+13\right)P^{(4)}_{n+1}(x)\notag\\
&&+8 (n+1) x \left(4 x^4+36 x^2+35\right)P^{(3)}_{n+1}(x)\notag\\
&& +8\left(-16 x^8+8 (n-27) x^6+8 (11 n-94) x^4+6 (29 n-111) x^2-105\right)P''_{n+1}(x) \notag\\
&&+16 (n+1) x \left(2 x^2+7\right) \left(4 x^4+36 x^2+15\right)P'_{n+1}(x)\notag\\
&&-16 (n+1) \left(2 x^2+7\right) \left(4 x^4+36 x^2+15\right) P_{n+1}(x)=0,\notag\\
%%%%%%%%%%%%%%%%%%%%%%%%%%%%%%%%%%%%%%%%%%%%%%%%
j=3, \  && \left(2 x^2+1\right)P^{(9)}_{n+1}(x) +2 (n+1) \left(2 x^2+1\right)P^{(7)}_{n+1}(x)+4 (n+1) x \left(2 x^2+1\right)P^{(6)}_{n+1}(x)\notag\\
&&4 (n+1) \left(2 x^2+1\right) \left(2 x^2+7\right)P^{(5)}_{n+1}(x)+8 (n+1) x \left(2 x^2+1\right) \left(2 x^2+13\right)P^{(4)}_{n+1}(x)\notag\\
&&+8 \left(8 (n-27) x^6+4 (19 n-191) x^4+2 (53 n-367) x^2+35 (n-2)-16 x^8\right)P^{(3)}_{n+1}(x)\notag\\
&&16 (n+1) x \left(2 x^2+1\right) \left(4 x^4+44 x^2+87\right)P''_{n+1}(x) \notag\\
&&-64 (n+1) x^2 \left(4 x^4+44 x^2+87\right)P'_{n+1}(x)\notag\\
&&+64 (n+1) x \left(4 x^4+44 x^2+87\right) P_{n+1}(x)=0,\notag\\
%%%%%%%%%%%%%%%%%%%%%%%%%%%%%%%%%%%%%%%%%%%%%%%%
 j=4, \  && x \left(2 x^2+3\right)P^{(9)}_{n+1}(x) +2 (n+1) x \left(2 x^2+3\right)P^{(7)}_{n+1}(x)\notag\\
 &&+4 (n+1) x^2 \left(2 x^2+3\right)P^{(6)}_{n+1}(x)+4 (n+1) x \left(2 x^2+3\right) \left(2 x^2+7\right)P^{(5)}_{n+1}(x) \notag\\
 &&+ 4 \left(8 (n-27) x^6+8 (8 n-97) x^4+6 (13 n-127) x^2-16 x-105\right)P^{(4)}_{n+1}(x)\notag\\
&&+8 (n+1) x \left(2 x^2+3\right) \left(4 x^4+36 x^2+35\right)P^{(3)}_{n+1}(x)\notag\\
&& -24 (n+1) \left(2 x^2+1\right) \left(4 x^4+36 x^2+35\right)P''_{n+1}(x) \notag\\
&&+96 (n+1) x \left(4 x^4+36 x^2+35\right)P'_{n+1}(x)\notag\\
&&-96 (n+1) \left(4 x^4+36 x^2+35\right)P_{n+1}(x)=0,\notag\\
%%%%%%%%%%%%%%%%%%%%%%%%%%%%%%%%%%%%%%%%%%%%%%%%
 j=5, \  && \left(4 x^4+12 x^2+3\right)P^{(9)}_{n+1}(x) +2 (n+1) \left(4 x^4+12 x^2+3\right)P^{(7)}_{n+1}(x)\notag\\
 &&+4 (n+1) x \left(4 x^4+12 x^2+3\right)P^{(6)}_{n+1}(x)\notag\\
 &&+4 \left(-16 x^8+8 (n-27) x^6+4 (13 n-197) x^4+30 (3 n-25) x^2+21 (n-4)\right)P^{(5)}_{n+1}(x) \notag\\
 &&+8 (n+1) x \left(2 x^2+13\right) \left(4 x^4+12 x^2+3\right)P^{(4)}_{n+1}(x)\notag\\
&& -64 (n+1) x^2 \left(2 x^2+3\right) \left(2 x^2+13\right)P^{(3)}_{n+1}(x)\notag\\
&& +192 (n+1) x \left(2 x^2+1\right) \left(2 x^2+13\right)P''_{n+1}(x) \notag\\
&&-768 (n+1) x^2 \left(2 x^2+13\right)P'_{n+1}(x)+768 (n+1) x \left(2 x^2+13\right)P_{n+1}(x)=0,\notag\\
%%%%%%%%%%%%%%%%%%%%%%%%%%%%%%%%%%%%%%%%%%%%%%%%
 j=6, \  &&x \left(4 x^4+20 x^2+15\right)P^{(9)}_{n+1}(x)+2 (n+1) x \left(4 x^4+20 x^2+15\right)P^{(7)}_{n+1}(x)\notag\\
 &&+2 \left(-16 x^8+8 (n-27) x^6+40 (n-20) x^4+30 (n-27) x^2-105\right)P^{(6)}_{n+1}(x)\notag\\
 &&+4 (n+1) x \left(2 x^2+7\right) \left(4 x^4+20 x^2+15\right)P^{(5)}_{n+1}(x) \notag\\
 && -20 (n+1) \left(2 x^2+7\right) \left(4 x^4+12 x^2+3\right)P^{(4)}_{n+1}(x)\notag\\
&& +160 (n+1) x \left(2 x^2+3\right) \left(2 x^2+7\right)P^{(3)}_{n+1}(x)\notag\\
&& -480 (n+1) \left(2 x^2+1\right) \left(2 x^2+7\right)P''_{n+1}(x) \notag\\
&&+1920 (n+1) x \left(2 x^2+7\right)P'_{n+1}(x)-1920 (n+1) \left(2 x^2+7\right)P_{n+1}(x)=0,\notag\\
%%%%%%%%%%%%%%%%%%%%%%%%%%%%%%%%%%%%%%%%%%%%%%%%
 j=7, \  && \left(8 x^6+60 x^4+90 x^2+15\right)P^{(9)}_{n+1}(x) \notag\\
 &&+2  \left(-16 x^8+8 (n-27) x^6+60 (n-13) x^4+30 (3 n-25) x^2+15 (n-6)\right)P^{(7)}_{n+1}(x)\notag\\
 &&+4 (n+1) x \left(8 x^6+60 x^4+90 x^2+15\right)P^{(6)}_{n+1}(x)\notag\\
 &&-48 (n+1) x^2 \left(4 x^4+20 x^2+15\right)P^{(5)}_{n+1}(x) \notag\\
 && +240 (n+1) x \left(4 x^4+12 x^2+3\right)P^{(4)}_{n+1}(x)-1920 (n+1) x^2 \left(2 x^2+3\right)P^{(3)}_{n+1}(x)\notag\\
&& +5760 (n+1) x \left(2 x^2+1\right)P''_{n+1}(x) -23040 (n+1) x^2P'_{n+1}(x)\notag\\
&&+23040 (n+1) xP_{n+1}(x)=0,\notag\\
%%%%%%%%%%%%%%%%%%%%%%%%%%%%%%%%%%%%%%%%%%%%%%%%
 j=8, \  && x \left(8 x^6+84 x^4+210 x^2+105\right)P^{(9)}_{n+1}(x) \notag\\
 &&-\left(16 x^8+224 x^6+840 x^4+840 x^2+105 \right)P^{(8)}_{n+1}(x)\notag\\
 &&+2 (n+1) x \left(8 x^6+84 x^4+210 x^2+105\right)P^{(7)}_{n+1}(x)\notag\\
 &&-14 (n+1) \left(8 x^6+60 x^4+90 x^2+15\right)P^{(6)}_{n+1}(x)\notag\\
 &&+168 (n+1) x \left(4 x^4+20 x^2+15\right)P^{(5)}_{n+1}(x) \notag\\
 && -840 (n+1) \left(4 x^4+12 x^2+3\right)P^{(4)}_{n+1}(x)+6720 (n+1) x \left(2 x^2+3\right)P^{(3)}_{n+1}(x)\notag\\
&& -20160 (n+1) \left(2 x^2+1\right)P''_{n+1}(x) +80640 (n+1) xP'_{n+1}(x)\notag\\
&&-80640 (n+1) P_{n+1}(x)=0.\notag
\end{eqnarray}

%%%%%%%%%%%%%%%%%%%%%%%%%%%%%%%%%%%%%%%%%%%%%%
%%%%%%%%%%%%%%%%%%%%%%%%%%%%%%%%%%%%%%%%%%%%%%
\subsection{Case~2 analogous to to Hermite}
%%%%%%%%%%%%%%%%%%%%%%%%%%%%%%%%%%%%%%%%%%%%%%
%%%%%%%%%%%%%%%%%%%%%%%%%%%%%%%%%%%%%%%%%%%%%%
\vspace{0.5cm}
The regularity conditions for this sequence are:
$
\lambda, \rho \in\mathbb{C}, \ \rho\neq 0.
$
The canonical form $u_0$ of this sequence is related to the classical Hermite form ${\cal H}$ by 
$$
u_{0}^{(1)}={\cal H}.
$$
 The input data for the algorithm {\it HoDELH} are as follows:
 
\noindent  - the recurrence coefficients, 

\begin{equation}\label{RC_LH0Hermite2}
\beta_{0}=\lambda,\quad \beta_{n+1}=0,\ n\geq 0;\qquad 
\gamma_{1}=\frac{\rho}{2},\quad \gamma_{n+1}=\frac{n}{2},\ n\geq 1.
\end{equation}
- the coefficients of the Stieltjes equation,
\begin{equation}\label{SE_LH0Hermite2}
 \Phi(x)=1,\quad B(x)=2x^2-2\lambda x +1-\rho,\quad  C(x)=2x,\quad D(x)=0.
 \end{equation}
- the coefficients of the Laguerre-Hahn structure relation, \cite{Article-2-Soummi}
\begin{equation}\label{SR_LH0Hermite2}
C_0(x)=C(x), \ D_0(x)=D(x)\ ;\quad C_{n+1}(x)=-2x,\quad  D_{n+1}(x)=-2,\quad n\geq 0.
\end{equation}

As in the preceding case, some coefficients appearing in the structure relations and the differential equations satisfy initial conditions, which can be expressed compactly using the notation
$$
\eta_n=n+\rho\delta_{n,0},\quad n\geq 0.
$$
However, the reduced coefficients of some of the differential equations, obtained after eliminating common factors, are free of initial conditions.

%%%%%%%%%%%%%%%%%%%%%%%%%%%%%%%%%%%%%%%%%%%%%%
\subsubsection{Structure relations}
%%%%%%%%%%%%%%%%%%%%%%%%%%%%%%%%%%%%%%%%%%%%%%

\vspace{0.25cm}

\begin{proposition}\label{Prop_SR_Hermite2}
If \(\{P_n\}_{n \geq 0}\) is the Laguerre-Hahn analogous to the Hermite Case 2 monic orthogonal polynomial sequence with recurrence coefficients given by \eqref{RC_LH0Hermite2}, then, for $n\geq 0$ and $i\geq 1$, the coefficients of the structure relation \eqref{ST_RR_i}-\eqref{RR_F} in Theorem \ref{th:0.1} satisfy
\begin{eqnarray}
&&G_{0,1}(x;n)=0,\label{G01_Hermite2}\\
&& G_{1,1}(x;n)=B_0(x)=2x^2-2\lambda x+1-\rho ,\label{G11_Hermite2}\\
&& G_{0,i+1}(x;n)=2xG_{0,i}(x;n)+\eta_n G_{1,i}(x,n)+B_0(x)H_i(x;n)+G'_{0,i}(x;n),\label{RR_G01_Hermite2}\\
&&G_{1,i+1}(x;0)=G'_{1,i}(x,0),\label{RRG10_Hermite2}\\
&& G_{1,i+1}(x;n)=-2G_{0,i}(x;n)+G'_{1,i}(x;n),\ n\geq 1,\label{RRG1_Hermite2}\\
&& H_{1}(x;n)=\eta_n, \label{H1_Hermite2}\\
&& H_{i+1}(x;n) = H'_i(x;n),\label{RR_Hi_CLHermite}\\
&& M_{0,1}(x;n)=2x, \label{M01_Hermite2}\\
%&& M_{0,2}(x;n)=2(\eta_n+1), \label{M02_Hermite2}\\
&&M_{0,i+1}(x;0) =M'_{0,i}(x;0), \ i\geq 2,\label{RR_M0i_Hermite2}\\
%%%%%%%%%%%%%%%%%%%%%%%%%%%%%%%%%%%
&& M_{0,i+1}(x;n)=M'_{0,i}(x;n)+2H_i(x;n),\ n\geq 1,\ i\geq 1, \label{RR_M0i1_Hermite2}\\
&&M_{k,i+1}(x;n) = M_{k-1,i}(x;n) +  M'_{k,i}(x;n),\  1 \le k \le i, \label{RR_Mk_Hermite2}\\
&& M_{i,i}(x;n) =1.  \label{RR_Mii_Hermite2}
\end{eqnarray}
Moreover,  for $n\geq 0$, 
\begin{eqnarray}
&&H_{i+1}(x;n)=0, \quad i\geq 1, \label{Hi_Hermite2}\\
&&M_{0,i+1}(x;n)=0,\quad i\geq 2, \label{M0i_Hermite2}\\
&& M_{i-1,i}(x;n)=2x,\quad i\geq 2, \label{Mi1i_Hermite2}\\
%&& M_{0,2}(x;n)=2(\eta_n+1),  \label{M02_Hermite2}\\
&&M_{i-2,i}(x;n)= 2(\eta_n+i-1) ,\quad i\geq 2,\label{Mi2i_Hermite2}\\
&&M_{i-j,i}(x;n)= 0, \quad j\geq 3 ,\quad i\geq j.\label{Miji_Hermite2}
\end{eqnarray}
\vspace{0.15cm}
Consequently, the structure relations take the form

\noindent - for  $i=1$,
\begin{eqnarray}
&&  \left(2x^2-2\lambda x +1-\rho\right) P_{n}^{(1)}(x) + \eta_n P_n(x)=
 P'_{n+1}(x)+2xP_{n+1}(x).\label{SR1_Hermite2}
\end{eqnarray}

\noindent - for  $i\geq 2$,
\begin{eqnarray}
&& G_{0,i}(x;n) P_{n-1}^{(1)}(x) + G_{1,i}(x;n) P_{n}^{(1)}(x) =
 P^{[i]}_{n+1}(x)+2xP^{[i-1]}_{n+1}(x)\label{SRi_Hermite2}\\
 &&+ 2(n+i-1)P^{[i-2]}_{n+1}(x).\notag
\end{eqnarray}
\end{proposition}
%%%%%%%%%%%%%%%%%%%%%%%

\begin{proof}
The formulas \eqref{G01_Hermite2}--\eqref{RR_Mii_Hermite2} are obtained by substituting the characteristic elements \eqref{RC_LH0Hermite2}, \eqref{SE_LH0Hermite2}, and \eqref{SR_LH0Hermite2} of this sequence into the corresponding formulas \eqref{G01}--\eqref{RR_Mii} of Theorem~\ref{th:0.1}.

It remains to prove the remaining formulas.

Applying recursively the recurrence relation \eqref{RR_Hi_CLHermite}, and taking into account \eqref{H1_Hermite2}, we obtain
$
H_{i+1}(x;n)=H_1^{(i)}(x;n)=(\eta_n)^{(i)}=0,\ i\ge1,
$
which proves \eqref{Hi_Hermite2}.

Next, applying \eqref{RR_M0i_Hermite2} with $i=1$ and using \eqref{M01_Hermite2}, we obtain
$
M_{0,2}(x;0)=M'_{0,1}(x;0)=2.
$
For $i\ge2$, applying \eqref{RR_M0i_Hermite2} yields
$
M_{0,i+1}(x;0)=M'_{0,i}(x;0)=M''_{0,i-1}(x;0).
$
We now prove by induction that $M_{0,i+1}(x;0)=0$ for all $i\ge2$. For $i=2$, we have
$
M_{0,3}(x;0)=M''_{0,1}(x;0)=0.
$
Assume that $M_{0,i+1}(x;0)=0$ for some fixed $i\ge2$. Then
$
M_{0,i+2}(x;0)=M'_{0,i+1}(x;0)=0,
$
which completes the proof of \eqref{M0i_Hermite2}.

We next prove \eqref{Mi1i_Hermite2} and \eqref{Mi2i_Hermite2} by induction.
Applying \eqref{RR_M0i1_Hermite2} with $i=1$, and using \eqref{M01_Hermite2} and \eqref{H1_Hermite2}, we obtain
$
M_{0,2}(x;n)=M'_{0,1}(x;n)+2H_1(x;n)
=2+2\eta_n
=2(\eta_n+1),
$
which is precisely \eqref{Mi2i_Hermite2}, corresponding to the case $i=2$.
Next, applying \eqref{RR_Mk_Hermite2} with $k=i=1$, and using \eqref{M01_Hermite2} together with \eqref{RR_Mii_Hermite2} for $i=1$, we obtain
$
M_{1,2}(x;n)=M_{0,1}(x;n)+M'_{1,1}(x;n)
=2x+0
=2x,
$
which proves \eqref{Mi1i_Hermite2} for $i=2$.
Now assume that \eqref{Mi1i_Hermite2} and \eqref{Mi2i_Hermite2} hold for some fixed $i\ge2$. Applying \eqref{RR_Mk_Hermite2} with $k=i$, we obtain
$
M_{i,i+1}(x;n)
=M_{i-1,i}(x;n)+M'_{i,i}(x;n)
=2x+0
=2x,
$
by the induction hypothesis and \eqref{RR_Mii_Hermite2}. Likewise, applying \eqref{RR_Mk_Hermite2} with $k=i-1$, we get
$
M_{i-1,i+1}(x;n)
=M_{i-2,i}(x;n)+M'_{i-1,i}(x;n)
=2(\eta_n+i-1)+2
=2(\eta_n+i),
$
by the induction hypothesis and \eqref{Mi2i_Hermite2}.

Let us now prove the last formula, \eqref{Miji_Hermite2}. We first consider the case $j=3$, namely $M_{i-3,i}=0$, $i\geq 3$.   It follows directly from \eqref{M0i_Hermite2} with $i=2$. Assume that the statement holds for some fixed $i\geq 3$. We show that it also holds for $i+1$. Setting $k=i-2$ in \eqref{RR_Mk_Hermite2}, we obtain $M_{i-2,i+1} = M_{i-3,i} +  M'_{i-2,i}$.
By the induction hypothesis, $M_{i-3,i}=0$ , while $M'_{i-2,i}=0$ follows from \eqref{Mi2i_Hermite2}. Hence $M_{i-2,i+1} =0$,
which completes the induction for the case $j=3$.

Now,  fix $j\geq 3$, and assume that 
\begin{equation}\label{Mil_Hermite2}
M_{i-l,i}(x;n)=0,\quad 3\leq l \leq j, \quad i\geq l.
\end{equation}
We prove that the statement also holds for $l=j+1$, that is, $M_{i-(j+1),i}=0$ for $i\geq j+1$. For the base case $i=j+1$, we have $M_{0,j+1}=0$, which follows from \eqref{M0i_Hermite2}.
Now assume that $M_{i-(j+1),i}=0$. Setting $k=i+1-(j+1)$ in \eqref{RR_Mk_Hermite2}, we obtain
$M_{i-j,i+1}=M_{i-j-1,i}+M'_{i-j,i}$. The first term vanishes by the induction hypothesis, while the second vanishes by 
\eqref{Mil_Hermite2} with $l=j$. Therefore, $M_{i-j,i+1}=0$, which completes the proof.
%%%%%%%%%%%%%%%%%%%%%%%
\end{proof}

Next, we present the coefficients $G_{0,i}(x;n)$ and $G_{1,i}(x;n)$ for $n\geq 0$ and $i=2,...,9$, obtained by applying the algorithm \textit{HoDELH} with symbolic computations. These results for $i=2,3$, and $4$ were first reported in \cite{Article-1-NA}.

\vspace{0.25cm}
%%%%%%%%%%%%%%%%%%%%%%%%%%%%%%%%%%%%%%%%%%%%%%%
\noindent {$\bullet$ \bf $i=2$} 
%%%%%%%%%%%%%%%%%%%%%%%%%%%%%%%%%%%%%%%%%%%%%%%
 \begin{eqnarray}
&&  G_{0,2}(x;n)  = 2\eta_n\left( 2 x^2-2x \lambda   - (\rho -1)\right),\quad G_{1,2}(x;n)  =2(2 x- \lambda).\notag
\end{eqnarray}

%%%%%%%%%%%%%%%%%%%%%%%%%%%%%%%%%%%%%%%%%%%%%%%
\noindent {$\bullet$ \bf $i=3$} 
%%%%%%%%%%%%%%%%%%%%%%%%%%%%%%%%%%%%%%%%%%%%%%%
\begin{eqnarray}
G_{0,3}(x;n)&= &2\eta_n\left( 4 x^3-4 x^2\lambda   -2 x(\rho -4) -3 \lambda  \right),\notag\\
G_{1,3}(x;n)&= & 4\left(-2 n x^2+2x \lambda  n + (n (\rho -1)+1)\right).\notag
\end{eqnarray}

%%%%%%%%%%%%%%%%%%%%%%%%%%%%%%%%%%%%%%%%%%%%%%%
\noindent {$\bullet$ \bf $i=4$} 
%%%%%%%%%%%%%%%%%%%%%%%%%%%%%%%%%%%%%%%%%%%%%%%
\begin{eqnarray}
G_{0,4}(x;n) & = &4\eta_n\left(
4 x^4-4 x^3\lambda   -2x^2  (n+\rho -7)+x\lambda   (2n-7) + (n (\rho -1)-\rho +5)\right),n\geq 1;\notag\\
G_{1,4}(x;n) & = & 4n\left(-4 x^3+4 \lambda   x^2+2 x(\rho -6) +5\lambda  \right).\notag
\end{eqnarray}

%%%%%%%%%%%%%%%%%%%%%%%%%%%%%%%%%%%%%%%%%%%%%%%
\noindent {$\bullet$ \bf $i=5$} 
%%%%%%%%%%%%%%%%%%%%%%%%%%%%%%%%%%%%%%%%%%%%%%%
\begin{eqnarray}
%G_{0,5}(x;0)  & = & 4\rho\left( 8 x^5-8 \lambda  x^4-4 (\rho -11) x^3-26 \lambda  x^2-2 (3 \rho-19) x-7 \lambda\right), \notag\\
G_{0,5}(x;n)  & = & 4\eta_n\left(8 x^5-8 \lambda  x^4-4 x^3(2 n+\rho -11) +2 \lambda  (4 n-13)x^2 \right. \notag\\
&&\left.+2  ( n( 2\rho-9 )-3 \rho +19)x+7 \lambda  (n-1)\right),\notag\\
G_{1,5}(x;n)  & = &8n\left(-4 x^4+4 \lambda x^3 +2 x^2 (n+\rho -10)-\lambda  (2 n-11) x+n(1- \rho)+2 \rho -11\right).\notag
\end{eqnarray}

%%%%%%%%%%%%%%%%%%%%%%%%%%%%%%%%%%%%%%%%%%%%%%%
\noindent {$\bullet$ \bf $i=6$}
%%%%%%%%%%%%%%%%%%%%%%%%%%%%%%%%%%%%%%%%%%%%%%%
\begin{eqnarray}
%G_{0,6}(x;0)  & = & 8\rho\left( 8 x^6-8 \lambda  x^5-4 (\rho -16) x^4-42 \lambda  x^3-4 (3 \rho -26)
%   x^2-33 \lambda  x-3 \rho +19\right), \notag\\
   %%%%%%%%%%%%%%%%%%%%%%%%%%%%%%%%%%%%%%%%%%%
G_{0,6}(x;n)  & = & 8\eta_n\left(8 x^6-8 \lambda  x^5-4 x^4 (3 n+\rho -16)+6 \lambda (2 n-7) x^3\right.\notag\\
&&+2 x^2 \left(n^2+ n(3 \rho -25)-6 \rho +52\right)-\lambda  \left(2 n^2-26 n+33\right) x\notag\\
   &&\left.-(n-1) (n (\rho -1)-3 \rho +19)\right),\notag\\
   %%%%%%%%%%%%%%%%%%%%%%%%%%%%%%%%%%%%%%%%%%%%
G_{1,6}(x;n)  & = &8n\left(-8 x^5+8 \lambda  x^4+4 x^3 (2 n+\rho -15)-2 \lambda  (4 n-19) x^2\right.\notag\\
&&\left.-2 x ( n(2 \rho-11) -5 \rho +39)-9 \lambda  (n-2)\right).\notag
\end{eqnarray}

%%%%%%%%%%%%%%%%%%%%%%%%%%%%%%%%%%%%%%%%%%%%%%%
\noindent {$\bullet$ \bf $i=7$}
%%%%%%%%%%%%%%%%%%%%%%%%%%%%%%%%%%%%%%%%%%%%%%%
\begin{eqnarray}
%G_{0,7}(x;0)  & = & \rho\left( 9\right), \notag\\
   %%%%%%%%%%%%%%%%%%%%%%%%%%%%%%%%%%%%%%%%%%%
G_{0,7}(x;n)  & = & 8\eta_n\left(16 x^7-16 \lambda  x^6-8 x^5 (4 n+\rho -22)+4 x^3 \left(3 n^2+ n (4\rho -52) -10 \rho +116\right)\right.\notag\\
&&-6 \lambda  \left(2 n^2-21n+32\right) x^2-2 x \left(n^2 (3 \rho -14)+n (109-15 \rho )+3 (5 \rho -41)\right)\notag\\
   &&\left.-11 \lambda  (n-3) (n-1)+4 \lambda  (8n-31) x^4  )\right),\notag\\
   %%%%%%%%%%%%%%%%%%%%%%%%%%%%%%%%%%%%%%%%%%%%
G_{1,7}(x;n)  & = &-16n\left( 8 x^6-8 \lambda  x^5-4 x^4 (3 n+\rho -21)+2 \lambda 
   (6 n-29) x^3\right.\notag\\
   &&+2 x^2 \left(n^2+n(3 \rho -31)-9 \rho +97\right)-\lambda  \left(2 n^2-34
   n+71\right) x\notag\\
   &&\left.-(n-2) (n \rho -n-4 \rho +29) \right).\notag
\end{eqnarray}

%%%%%%%%%%%%%%%%%%%%%%%%%%%%%%%%%%%%%%%%%%%%%%%
\noindent {$\bullet$ \bf $i=8$}
%%%%%%%%%%%%%%%%%%%%%%%%%%%%%%%%%%%%%%%%%%%%%%%
\begin{eqnarray}
%G_{0,8}(x;0)  & = & \rho\left( 9\right), \notag\\
   %%%%%%%%%%%%%%%%%%%%%%%%%%%%%%%%%%%%%%%%%%%
G_{0,8}(x;n)  & = & -16\eta_n\left(-16 x^8+16 \lambda  x^7+8 x^6 (5 n+\rho -29)-4 \lambda  (10 n-43)
   x^5\right.\notag\\
  && -4 x^4 \left(6 n^2+ n(5 \rho -93)-15 \rho +226\right)+8 \lambda  \left(3
   n^2-31 n+55\right) x^3\notag\\
   &&+2 x^2 \left(n^3+6 n^2 (\rho -9)-2 n (18 \rho -181)+3 (15 \rho -157)\right)\notag\\
   &&\left.-\lambda  \left(2 n^3-57 n^2+241 n-225\right) x-(n-3)
   (n-1) (n (\rho -1)-5 \rho +41)  )\right),\notag\\
   %%%%%%%%%%%%%%%%%%%%%%%%%%%%%%%%%%%%%%%%%%%%
G_{1,8}(x;n)  & = &-16n\left( 16 x^7-16 \lambda  x^6-8 x^5 (4 n+\rho -28)+4 \lambda  (8
   n-41) x^4\right.\notag\\
   &&+4 x^3 \left(3 n^2+4 n (\rho -16)-2 (7 \rho -100)\right)\notag\\
   &&-6 \lambda  \left(2 n^2-27n+61\right) x^2-2 x \left(n^2 (3 \rho -16)-3 n (7 \rho -57)+33 \rho -317\right)\notag\\
   &&\left.-13 \lambda  (n-4) (n-2) \right).\notag
\end{eqnarray}

%%%%%%%%%%%%%%%%%%%%%%%%%%%%%%%%%%%%%%%%%%%%%%%
\noindent {$\bullet$ \bf $i=9$}
%%%%%%%%%%%%%%%%%%%%%%%%%%%%%%%%%%%%%%%%%%%%%%%
\begin{eqnarray}
%G_{0,9}(x;0)  & = & \rho\left( 9\right), \notag\\
   %%%%%%%%%%%%%%%%%%%%%%%%%%%%%%%%%%%%%%%%%%%
G_{0,9}(x;n)  & = & -16\eta_n\left( -32 x^9+32 \lambda  x^8+16 x^7 (6
   n+\rho -37)-24 \lambda  (4 n-19) x^6\right.\notag\\
   &&-8 x^5 \left(10 n^2+n (6 \rho -151)-21 \rho +400\right)+20 \lambda  \left(4
   n^2-43 n+87\right) x^4\notag\\
   &&+4 x^3 \left(4 n^3+2 n^2 (5 \rho -71)-2 n (35 \rho -467)+5 (21 \rho -275)\right)\notag\\
   &&-2 \lambda  \left(8 n^3-174 n^2+796
   n-885\right) x^2\notag\\
   &&-2 x \left(n^3 (4 \rho -19)-6 n^2 (7 \rho -54)+8 n (16 \rho -151)-15 (7 \rho -71)\right)\notag\\
   &&\left.-15 \lambda  (n-5) (n-3) (n-1) )\right),\notag\\
   %%%%%%%%%%%%%%%%%%%%%%%%%%%%%%%%%%%%%%%%%%%%
G_{1,9}(x;n)  & = &32n\left(-16 x^8+16 \lambda  x^7+8 x^6 (5 n+\rho -36)-20 \lambda  (2
   n-11) x^5 \right.\notag\\
   &&-4 x^4 \left(6 n^2+n (5 \rho -113)-2 (10 \rho -183)\right)+24 \lambda 
   \left(n^2-13 n+32\right) x^3\notag\\
   &&+2 x^2 \left(n^3+3 n^2 (2 \rho -21)-2 n (24 \rho -277)+3 (29 \rho -357)\right)\notag\\
   &&\left.-\lambda  \left(2 n^3-69 n^2+403 n-591\right)x-(n-4) (n-2) (n (\rho -1)-6 \rho +55) \right).\notag
\end{eqnarray}

%%%%%%%%%%%%%%%%%%%%%%%%%%%%%%%%%%%%%%%%%%%%%%
\subsubsection{Differential equations}
%%%%%%%%%%%%%%%%%%%%%%%%%%%%%%%%%%%%%%%%%%%%%%

Next, we present the reduced coefficients $\widehat{\mathcal{K}}_d(x;n)$, $0\leq d \leq N$, for a differential equation of order $N=5$  and another of order $N=10$, together with the greatest common factor $c(x;n)$ between the coefficients $\mathcal{K}_d(x;n)$,
$0\leq d \leq N$, for Case 2 of the Laguerre-Hahn family analogous to Hermite. These results were obtained by symbolic computations using the algorithm \textit{HoDELH}.

\vspace{0.5cm}

%%%%%%%%%%%%%%%%%%%%%%%%%%%%%%%%%%%%%%%%%%%%%%%
\noindent  {\bf $\bullet$ Fifth-order homogeneous linear differential equation for} 
\newline ${\bf (i_1,i_2,i_3,i_4)=(1,2,3,5=N)}$:
%%%%%%%%%%%%%%%%%%%%%%%%%%%%%%%%%%%%%%%%%%%%%%%

\vspace{0.25cm}

\noindent {\bf Greatest common factor:}
\begin{eqnarray}
 c(x;0)=4 \rho ^2,\quad  c(x;n)= 4n^2,\quad n\geq 1. \notag
\end{eqnarray}

%\vspace{0.25cm}

\noindent  {\bf Reduced coefficients $\widehat{\mathcal{K}}_d(x;n)$, $0\leq d \leq 5=N$, for $n\geq 0$ :}

\begin{eqnarray}
%%%%%%%%%%%%%%%%%%%%%%%%
\widehat{\mathcal{K}}_5(x;n)& = &  -8 (n+1) x^4+4 \lambda  (4 n+3) x^3-4 x^2
   \left(\lambda ^2+2 n \left(\lambda ^2-\rho +1\right)-\rho +3\right)\notag\\
   &&-2 \lambda  x
   (4 n (\rho -1)+\rho -5)-2 n (\rho -1)^2-3 \lambda ^2-2 \rho +2,\notag\\
%\end{eqnarray}
%%%%%%%%%%%%%%%%%%%%%%%%
%\begin{eqnarray}
%%%%%%%%%%%%%%%%%%%%%%%%
\widehat{\mathcal{K}}_4(x;n)& = &-2x\left( 8 (n+1) x^4-4 \lambda  (4 n+3) x^3+4 x^2
   \left(2 n \left(\lambda ^2-\rho +1\right)+\lambda ^2-\rho +3\right)\right.\notag\\
   &&\left.+2 \lambda  x(4 n (\rho -1)+\rho -5)+2 n (\rho -1)^2+3 \lambda ^2+2 \rho -2\right) ,\notag\\
%%%%%%%%%%%%%%%%%%%%%%%%
%\end{eqnarray}
%%%%%%%%%%%%%%%%%%%%%%%%
%\begin{eqnarray}
%%%%%%%%%%%%%%%%%%%%%%%%
\widehat{\mathcal{K}}_3(x;n)& = &-2\left(-16 (n+1) x^6+8 \lambda  (4 n+3) x^5+8 x^4 \left(2 n^2-2 n \left(\lambda ^2-\rho +2\right)-\lambda ^2+\rho -7\right)\right.\notag\\
&&-4 \lambda  x^3 \left(8 n^2+4 n (\rho -2)+\rho -12\right)\notag\\
&&+2 x^2 \left(-5 \lambda ^2+8 n^2 \left(\lambda ^2-\rho +1\right)-2 n (\rho -1)^2-2 \rho
   -14\right)\notag\\
   &&+2 \lambda  x\left(8 n^2 (\rho -1)+4 n (\rho -1)+\rho +6\right) \notag\\
   &&\left.4 n^2 (\rho -1)^2+2 n \left(5 \lambda ^2+2 \rho ^2+6 \rho
   -8\right)+5 \lambda ^2+8 \rho -8\right) ,\notag\\
%%%%%%%%%%%%%%%%%%%%%%%%
%\end{eqnarray}
%%%%%%%%%%%%%%%%%%%%%%%%
%\begin{eqnarray}
%%%%%%%%%%%%%%%%%%%%%%%%
\widehat{\mathcal{K}}_2(x;n)& = & -4\left(-16 (n+1) x^7+8 \lambda  (4 n+3) x^6\right.\notag\\
&&+8 x^5 \left(2 n^2+n \left(-2 \lambda ^2+2 \rho -7\right)-\lambda ^2+\rho -10\right)\notag\\
&&-4 \lambda  x^4\left(8 n^2+2 n (2 \rho -11)+\rho -22\right)\notag\\
&&+2 x^3 \left(8 n^2 \left(\lambda ^2-\rho +2\right)-2 n \left(8 \lambda
   ^2+\rho ^2-10 \rho -7\right)-13 \lambda ^2+4 \rho -16\right)\notag\\
   &&+2 \lambda  x^2 \left(4 n^2 (2 \rho -5)+n (-14 \rho -17)-3 \rho +13\right)\notag\\
   &&+x \left(4 n^2 \left(2 \lambda ^2+\rho ^2-4 \rho +3\right)+2 n \left(13
   \lambda ^2-3 \rho ^2+18 \rho +5\right)-7 \lambda ^2-6 \rho +38\right)\notag\\
   &&\left.+\lambda  \left(4 n^2 (\rho -1)+n (-\rho -19)-7\right)\right) ,\notag\\
%%%%%%%%%%%%%%%%%%%%%%%%
%\end{eqnarray}
%%%%%%%%%%%%%%%%%%%%%%%%
%\begin{eqnarray}
%%%%%%%%%%%%%%%%%%%%%%%%
\widehat{\mathcal{K}}_1(x;n)& = &-4\left( -16 (n+1) x^6+16 \lambda  (3 n+2) x^5\right.\notag\\
&&+8 x^4 \left(n^3-n^2-2 n \left(2 \lambda ^2-2 \rho +5\right)-2 \lambda ^2+\rho -11\right)\notag\\
&&-4 \lambda  x^3 \left(4 n^3-3 n^2+2 n (5 \rho -22)+2 \rho -35\right)\notag\\
&&+4 x^2 \left(2 n^3 \left(\lambda ^2-\rho +1\right)+n^2 \left(-\lambda
   ^2+\rho -3\right)\right.\notag\\
   &&\left.-3 n \left(5 \lambda ^2+\rho ^2-10 \rho +9\right)-13 \lambda ^2+3 \rho
   -19\right)\notag\\
   &&+2 \lambda  x
   \left(4 n^3 (\rho -1)+n^2 (5-\rho )+n (41-30 \rho )-3 (2 \rho -15)\right)\notag\\
   &&\left.+(n+2) \left(2 n^2 (\rho -1)^2+n \left(7 \lambda ^2-4 \rho ^2+26 \rho
   -22\right)-7 \lambda ^2\right)\right),\notag\\
%%%%%%%%%%%%%%%%%%%%%%%%
%\end{eqnarray}
%%%%%%%%%%%%%%%%%%%%%%%%
%\begin{eqnarray}
%%%%%%%%%%%%%%%%%%%%%%%%
\widehat{\mathcal{K}}_0(x;n)& = & -8(n+1)\left( 8 (n+1)^2 x^5-4 \lambda x^4 \left(4n^2+5 n+2\right) \right.\notag\\
&&+4 x^3 \left(2 n^2 \left(\lambda ^2-\rho +3\right)+n \left(\lambda ^2-\rho
   +19\right)-\rho +11\right)\notag\\
   &&+2 \lambda  x^2 \left(4 n^2 (\rho -4)+n (-\rho -34)-13\right)\notag\\
   &&+x \left(2 n^2 \left(4 \lambda ^2+\rho ^2-6 \rho +5\right)+n \left(19 \lambda ^2-2
   \rho ^2+10 \rho +32\right)-2 (3 \rho -19)\right)\notag\\
   &&\left.+\lambda  \left(4 n^2 (\rho -1)+n (-\rho -19)-7\right)\right),\notag
%%%%%%%%%%%%%%%%%%%%%%%%
\end{eqnarray}

\vspace{0.25cm}

%%%%%%%%%%%%%%%%%%%%%%%%%%%%%%%%%%%%%%%%%%%%%%%
%%%%%%%%%%%%%%%%%%%%%%%%%%%%%%%%%%%%%%%%%%%%%%%
\noindent  {\bf $\bullet$ Tenth-order homogeneous linear differential equations for} 
\newline${\bf (i_1,i_2,i_3,i_4)=(1,8,9,10=N)}$:
%%%%%%%%%%%%%%%%%%%%%%%%%%%%%%%%%%%%%%%%%%%%%%%
%%%%%%%%%%%%%%%%%%%%%%%%%%%%%%%%%%%%%%%%%%%%%%%
\vspace{0.25cm}

\noindent {\bf Greatest common factor:}
\begin{eqnarray}
c(x;n)= 256 n^3 \left(n^4-10 n^3+35 n^2-50 n+24\right),\ n\geq 0. \notag
\end{eqnarray}

%\vspace{0.25cm}

\noindent  {\bf Reduced coefficients $\widehat{\mathcal{K}}_d(x;n)$, $0\leq d \leq 10=N$, for $n\geq 0$ :}

\begin{eqnarray}
%%%%%%%%%%%%%%%%%%%%%%%%
\widehat{\mathcal{K}}_{10}(x;n)& = &  -8 (n+1)^2 x^4+4 \lambda  \left(4 n^2-5 n-3\right)
   x^3\notag\\
   &&-4 x^2 \left(2 n^2 \left(\lambda ^2-\rho +1\right)+n \left(-9 \lambda ^2+9
   \rho +101\right)-5 \lambda ^2-31 \rho +57\right)\notag\\
   &&-2 \lambda  x \left(4 n^2 (\rho -1)+n (-31 \rho -189)+55 \rho +433\right)\notag\\
   &&-2 n^2 (\rho -1)^2+n \left(-195 \lambda ^2+22 \rho ^2-214 \rho
   +192\right)\notag\\
   &&+975 \lambda ^2-60 \rho ^2+1042 \rho -4510,\notag\\
%   \end{eqnarray}
%%%%%%%%%%%%%%%%%%%%%%%%
%%%%%%%%%%%%%%%%%%%%%%%%
%\begin{eqnarray}
\widehat{\mathcal{K}}_9(x;n)& = &2\left(16 (n+1)^2 x^3+-6 \lambda  \left(4 n^2-5 n-3\right) x^2\right.\notag\\
&&+4 x \left(2 n^2 \left(\lambda ^2-\rho +1\right)+n \left(-9 \lambda ^2+9
   \rho +101\right)-5 \lambda ^2-31 \rho +57\right)\notag\\
   &&\left.+\lambda  \left(4 n^2 (\rho -1)+n (-31 \rho -189)+55 \rho +433\right) \right) ,\notag\\
%      \end{eqnarray}
%%%%%%%%%%%%%%%%%%%%%%%%
%%%%%%%%%%%%%%%%%%%%%%%%
%\begin{eqnarray}
\widehat{\mathcal{K}}_8(x;n)& = &-2\left(-16 (n+1)^2 x^6+8 \lambda  \left(4 n^2-5 n-3\right) x^5\right.\notag\\
&&
   +8x^4 \left(2 n^3-2 n^2 \left(\lambda ^2-\rho -1\right)+9 n \left(\lambda
   ^2-\rho -11\right)+5 \lambda ^2+31 \rho -577\right)\notag\\
   && -4 \lambda  x^3 \left(8 n^3+4 n^2 (\rho -3)+n (-31 \rho -191)+5 (11 \rho +87)\right)\notag\\
  &&  +2 x^2 \left(8 n^3 \left(\lambda ^2-\rho +1\right)-2 n^2 \left(16 \lambda
   ^2+\rho ^2-18 \rho -235\right)\right.\notag\\
   &&\left.+n \left(-233 \lambda ^2+22 \rho ^2-346
   \rho +556\right)+965 \lambda ^2-60 \rho ^2+1038 \rho -44422\right)\notag\\
   && +2 \lambda  x \left(8 n^3 (\rho -1)-56 n^2 (\rho +8)+n (70 \rho +851)+50 \rho -11\right)\notag\\
   && 
  +4 n^3 (\rho -1)^2+2 n^2 \left(225 \lambda ^2-20 \rho ^2+238
   \rho -218\right)\notag\\
   &&+n \left(-2025 \lambda ^2+76 \rho ^2-2128 \rho
   +11852\right)\notag\\
 &&  \left.  -1125 \lambda ^2+120 \rho ^2-1268 \rho +1540 \right) ,\notag\\
 %   \end{eqnarray}
%%%%%%%%%%%%%%%%%%%%%%%%
%%%%%%%%%%%%%%%%%%%%%%%%
%\begin{eqnarray}
\widehat{\mathcal{K}}_7(x;n)& = & 4\left(104 (n+1)^2 x^5   -8 \lambda  \left(27 n^2-37 n-22\right) x^4\right.\notag\\
&& +4 x^3 \left(4 n^3+4 n^2 \left(7 \lambda ^2-7 \rho +9\right)+n \left(-126
   \lambda ^2+139 \rho +1515\right) \right.\notag\\
   &&\left. -70 \lambda ^2-463 \rho +853\right)\notag\\
   &&-2 \lambda  x^2 \left(12 n^3+n^2 (75-58 \rho )+24 n (19 \rho +120)-830 \rho -7011\right)\notag\\
   &&+2 x \left(4 n^3 \left(\lambda ^2-\rho +1\right)+n^2 \left(-16 \lambda
   ^2+15 \rho ^2-72 \rho +277\right)\right.\notag\\
   &&+n \left(1541 \lambda ^2-165 \rho
   ^2+1923 \rho -1318\right)\notag\\
   && \left.-7805 \lambda ^2+450 \rho ^2-8719 \rho +38477\right)\notag\\
   &&\left.+\lambda  (n+8) \left(4 n^2 (\rho -1)+n (-31 \rho -189)+55 \rho +433\right)
\right),\notag\\
%   \end{eqnarray}
%%%%%%%%%%%%%%%%%%%%%%%%
%%%%%%%%%%%%%%%%%%%%%%%%
%\begin{eqnarray}
\widehat{\mathcal{K}}_6(x;n)& = &-4 (n-5) (n+7)\left( 8 (n+1)^2 x^4-4 \lambda  \left(4 n^2-7
   n-4\right) x^3\right.\notag\\
   &&+4 x^2 \left(2 n^2 \left(\lambda ^2-\rho +1\right)+n \left(-11 \lambda
   ^2+11 \rho +131\right)-6 \lambda ^2-43 \rho +73\right)\notag\\
   &&+2 \lambda  x
   \left(4 n^2 (\rho -1)+n (-37 \rho -247)+78 \rho +695\right)\notag\\
   &&+2 n^2 (\rho -1)^2+n
   \left(255 \lambda ^2-26 \rho ^2+278 \rho -252\right)\notag\\
   &&\left.-2 \left(765 \lambda ^2-42 \rho ^2+811 \rho -3905\right)\right) ,\notag\\
%%%%%%%%%%%%%%%%%%%%%%%%
%%%%%%%%%%%%%%%%%%%%%%%%
\widehat{\mathcal{K}}_i(x;n)& = &0,\ i=0,...,5.\notag
%%%%%%%%%%%%%%%%%%%%%%%%
\end{eqnarray}

%%%%%%%%%%%%%%%%%%%%%%%%%%

\section*{Conclusions}
The main achievement of this work is to show that the method developed in \cite{Article-1-NA} for deriving the fourth-order differential equation of Laguerre--Hahn orthogonal polynomials is not an isolated result, but the first term of a complete hierarchy of differential equations of arbitrary order. This reveals a recursive structure inherent to the Laguerre--Hahn class, which had remained unnoticed until now. The case $N=4$, which has been the subject of numerous studies over the past decades, appears here as a particular instance of a much more general phenomenon.

From an algorithmic point of view, the method is fully constructive and its implementation in {\it Mathematica$^{\circledR}$}  makes it possible to compute these equations for any family of the Laguerre--Hahn class. The applications presented in this paper, particularly the results for the class-zero Hermite-analogous families and the classical Hermite sequence, demonstrate the practical relevance and efficiency of the approach. 

Further work is currently in progress to illustrate the method on other families and to explore the algebraic structure of the hierarchy.

In conclusion, this work establishes a new and unifying perspective on the differential equations of Laguerre--Hahn orthogonal polynomials, showing that the theory naturally extends to a complete hierarchy of equations of all orders, thereby opening new directions for both theoretical and computational research in the field.

%%%%%%%%%%%%%%%%%%%%%%%%%%%%%%%%%%%%%%%%%%%%%%%%%%%%

%%%%%%%%%%%%%%%%%%%%%%%%%%%%%%%%%%%%%%%%%%%%%%%%%%%%%%%%%%%%%%%%

\section*{Declarations}

\noindent  {\bf Data Availability} No datasets were generated or analysed during the current study.

\noindent  {\bf Conflicts of Interest} The authors have no conflicts of interest to declare.

\noindent  {\bf Competing interests} The authors declare no competing interests.

\noindent {\bf Funding:} The third author, Zélia da Rocha,  was partially supported by CMUP, a member of LASI, which is financed by national funds through FCT -- Funda\c c\~ao para a Ci\^encia e a Tecnologia, I.P., under the projects with reference UID/00144/2025 and associated DOI given by \url{https://doi.org/10.54499/UID/00144/2025}.

%\noindent {\bf Acknowledgment:} The authors would like to express their sincere gratitude to the referee for the careful and thorough review of the manuscript. The referee’s constructive comments have contributed to improving the paper.

\bibliographystyle{plain}
\bibliography{sn-bibliography-4oDELH}

%%%%%%%%%%%%%%%%%%%%%%%%%%%%%%%%%%%%%%%%%%%%%%%%%%%%%%%%
%%%%%%%%%%%%%%%%%%%%%%%%%%%%%%%%%%%%%%%%%%%%%%%%%%%%%%%%
\end{document}